\documentclass[leqno,11pt]{amsart}
\usepackage{amssymb, amsmath,amsmath,latexsym,amssymb,amsfonts,amsbsy, amsthm,esint}
\usepackage{color}
\usepackage{graphicx}
\usepackage{tikz}
\usetikzlibrary{arrows, calc, decorations.markings,intersections}
\usepackage{caption}
\usepackage{subcaption}
\usepackage{hyperref}
\usepackage[foot]{amsaddr}

\makeatletter
\DeclareFontFamily{U}{tipa}{}
\DeclareFontShape{U}{tipa}{m}{n}{<->tipa10}{}
\newcommand{\arc@char}{{\usefont{U}{tipa}{m}{n}\symbol{62}}}%

\newcommand{\arc}[1]{\mathpalette\arc@arc{#1}}

\newcommand{\arc@arc}[2]{%
  \sbox0{$\m@th#1#2$}%
  \vbox{
    \hbox{\resizebox{\wd0}{\height}{\arc@char}}
    \nointerlineskip
    \box0
  }%
}
\makeatother

\let\pa=\partial
\let\al=\alpha

\let\d=\delta

\let\lam=\lambda
\let\r=\rho
\let\s=\sigma
\let\f=\frac

\let\D=\Delta

\let\Om=\Omega

\let\e=\varepsilon
\let\pa=\partial

\let\ri=\rightarrow
\let\na=\nabla

\def\non{\nonumber}

\def\di{\mathrm{div}\,}

\newcommand{\beq}{\begin{equation}}
\newcommand{\eeq}{\end{equation}}
\newcommand{\beqo}{\begin{equation*}}
\newcommand{\eeqo}{\end{equation*}}
\newcommand{\ben}{\begin{eqnarray}}
\newcommand{\een}{\end{eqnarray}}
\newcommand{\beno}{\begin{eqnarray*}}
\newcommand{\eeno}{\end{eqnarray*}}

\newtheorem{thm}{Theorem}[section]

\newtheorem{theorem}{Theorem}[section]
\newtheorem{definition}[theorem]{Definition}
\newtheorem{lemma}[theorem]{Lemma}
\newtheorem{proposition}[theorem]{Proposition}

\theoremstyle{remark}

\newtheorem{rmk}{Remark}[section]

\newcommand{\dist}{\mathrm{dist}}
\newcommand{\BP}{\mathbb{P}}
\newcommand{\BR}{\mathbb{R}}
\newcommand{\BS}{\mathbb{S}^2}

\newcommand{\cn}{\mathcal{N}}
\newcommand{\qs}{Q^\sharp}

\begin{document}
\title[Uniform profile near point defects]{Uniform profile near the point defect\\
 of Landau-de Gennes model}

\author{Zhiyuan Geng$^1$}
\address{$^1$Basque Center for Applied Mathematics, Alameda de Mazarredo 14
48009 Bilbao, Bizkaia, Spain}
\email{zgeng@bcamath.org}

\author{Arghir Zarnescu$^{1,2,3}$}
\address{$^2$IKERBASQUE, Basque Foundation for Science, Plaza Euskadi 5 48009 Bilbao, Bizkaia, Spain}
\address{$^3$``Simion Stoilow" Institute of the Romanian Academy, 21 Calea Grivi\c{t}ei, 010702 Bucharest, Romania}
\email{azarnescu@bcamath.org}

\thanks{This research is supported by the Basque Government through the BERC 2022-2025 program and by the Spanish State Research Agency through BCAM Severo Ochoa excellence accreditation SEV-2017-0718 and through project PID2020-114189RB-I00 funded by Agencia Estatal de Investigaci\'{o}n (PID2020-114189RB-I00 / AEI / 10.13039/501100011033).}

\date{\today}

\begin{abstract}
For the Landau-de Gennes functional on 3D domains,
\beqo
I_\e(Q,\Omega):=\int_{\Om}\left\{\f12|\na Q|^2+\f{1}{\e^2}\left( -\f{a^2}{2}\mathrm{tr}(Q^2)-\f{b^2}{3}\mathrm{tr}(Q^3)+\f{c^2}{4}[\mathrm{tr}(Q^2)]^2 \right) \right\}\,dx,
\eeqo
it is well-known that under suitable boundary conditions, the global minimizer $Q_\e$ converges strongly in $H^1(\Om)$ to a uniaxial minimizer $Q_*=s_+(n_*\otimes n_*-\f13\mathrm{Id})$ up to some subsequence $\e_n\ri\infty$ , where $n_*\in H^1(\Om,\mathbb{S}^2)$ is a minimizing harmonic map. In this paper we further investigate the structure of $Q_\e$ near the core of a point defect $x_0$ which is a singular point of the map $n_*$. The main strategy is to study the blow-up profile of $Q_{\varepsilon_n}(x_n+\varepsilon_n y)$ where $\{x_n\}$ are carefully chosen and converge to $x_0$. We prove that $Q_{\varepsilon_n}(x_n+\varepsilon_n y)$ converges in $C^2_{loc}(\mathbb{R}^n)$ to a tangent map $Q(x)$ which at infinity behaves like a ``hedgehog" solution that coincides with the asymptotic profile of $n_*$ near $x_0$. Moreover, such convergence result implies that the minimizer $Q_{\varepsilon_n}$ can be well approximated by the Oseen-Frank minimizer $n_*$ outside the $O(\varepsilon_n)$ neighborhood of the point defect.
\end{abstract}

\maketitle

\section{Introduction}
Nematic liquid crystals (NLC) are composed of rigid rod-like molecules which exhibit a locally preferred direction. Sharp variations in the alignment direction of NLC are known as defects, which are generally observed, in experiments, to exist as isolated points or disclination lines in experiments. There are several continuum theories used to describe the local orientation of NLC molecules at equilibrium. In these theories, NLC materials are assumed to occupy a region $\Omega\in \BR^d\  (d=2,3)$ and their locally preferred directions are described by functions taking values in some order parameter spaces. The study of variational problems for energy-minimizing configuration of NLC (especially the configuration near defects) within these theories provides many fascinating mathematical problems. Readers are referred to survey articles \cite{Ball,Lin-Liu,z2021} and references therein for more details.

Among these theories, the simplest one is the Oseen-Frank theory \cite{frank}. In the Oseen-Frank theory, the local orientation of NLC is represented by a unit-vector field $n:\Omega\ri \mathbb{S}^2$, which minimizes an elastic energy. In the simplest setting, the free energy reduces to
\beqo
\int_{\Om} \f12 |\na n|^2 \,dx,
\eeqo
which is the energy functional for harmonic maps. The singular set of a  minimizing harmonic map is very well understood. In particular, in three dimensional space, the singular set contains at most finitely many points \cite{su}. Near each singularity $x_0$, the field $n$ behaves like the rotated ``hedgehog" map $\pm R\f{x-x_0}{|x-x_0|}$ with some rotation $R$ \cite{bcl}. We recall that the major limitations of the Oseen-Frank model are that it only accounts for uniaxial nematic states and does not allow for line defects of finite energy (see \cite{hkl}).

In the physically more realistic Landau-de Gennes theory\cite{deGennes}, the order parameter is a $3\times 3$ symmetric traceless matrix $Q$ (the so-called Q-tensors), which can be interpreted as the renormalized second moment of the (formal) probability distribution of the local molecular orientation. The total free energy contains two parts, namely the elastic energy and the bulk potential, whose simplified form reads
\beqo
\begin{split}
I_\e(Q,\Om):=&\int_{\Om} \left\{ f_e(Q,\na Q)+f_b(Q) \right\}\,dx\\
=&\int_{\Omega} \left\{ \f12|\na Q|^2+\f{1}{\e^2}(-\f{a^2}{2}\mathrm{tr}(Q^2)-\f{b^2}{3}\mathrm{tr}(Q^3)+\f{c^2}{4}[\mathrm{tr}(Q^2)]^2+C) \right\}\,dx,
\end{split}
\eeqo
where $\e,a,b,c$ are material dependent constants, $C$ is a constant that ensures $f_b(Q)\geq 0$. The Landau-de Gennes theory can predict richer and more complicated local behaviors of the NLC medium because it accounts for both uniaxial and biaxial phases ($Q$ is called biaxial when it has three distinct eigenvalues, uniaxial when it has only two equal eigenvalues and isotropic when all the three equal eigenvalues are zero)). In particular it allows biaxiality in the cores of point defects and disclination lines. Interested readers can refer to \cite{bpp,insz1,insz2,gm,canevari1,abl,di,insz3,insz4,canevari2,cl2017,hmp,dmp1,dmp2,acs,ty} for various studies on solutions and defect patterns of the Landau-de Gennes model.

When $\e\ri 0$, the Landau-de Gennes energy will enforce the uniaxial constraint $Q=s_+(n\otimes n-\f13\mathrm{Id})$ (so that the potential function takes its minimal value, see \eqref{def of min manifold}) and one can recover the Oseen-Frank model. Such convergence, which is usually referred to as the \emph{vanishing elasticity limit} (see \cite{gartland}), was first analysed in \cite{mz} and refined later on in \cite{nz}. Their results can be briefly summarized as follows: under suitable assumptions on the domain $\Omega$ and the boundary condition $Q|_{\pa\Om}$, the global minimizers $Q_\e$ converges strongly in $H^1$ to a limiting uniaxial minimizer $Q_*=s_+(n_*\otimes n_*-\f13\mathrm{Id})$ up to a subsequence, where $n_*\in H^1(\Omega,\mathbb{S}^2)$ is a minimizing harmonic map. Moreover, the convergence is strong in $C^k_{loc}(\Omega\setminus \mathcal{S}(n_*))$ for any non-negative integer $k$, where $\mathcal{S}(n_*)$ denotes the singular set of $n_*$. Similar limiting problems for Landau-de Gennes model have also been explored in \cite{bpp,gm,gms,canevari1,canevari2,clr} under various settings. The study of vanishing elasticity limits is influenced by similar analyses of the Ginzburg-Landau model for superconductors \cite{bbh93,bbh94}, while the higher dimension of the target space generates greater complexity in analysis for Q-tensors.

The main purpose of this paper is to further investigate the structure of minimizers $Q_\e$ in the core of a point defect $x_0\in\mathcal{S}(n_*)$ by studying the blow-up profile of $Q_{\e_n}(x_n+\e_ny)$ where $x_n$ will be carefully chosen and converge to $x_0$. We summarise our main results in the following theorem:
\begin{theorem}
Suppose $Q_{\e_n}$ is a sequence of global minimizers of $I_{\e_n}(\cdot,\Om)$ subjected to the Dirichlet boundary condition \eqref{bdy-con} and $Q_{\varepsilon_n}$ converges to the vanishing elasticity limit $Q_*$ in the sense of \cite{mz,nz}. Let $x_0\in\mathcal{S}(n_*)$. There exists a subsequence of $Q_{\e_n}$, denoted as itself, and a sequence $x_n\ri x_0$ such that the following holds

\begin{itemize}
 \item (Proposition \ref{oqconverge})$Q_{\e_n}(x_n+\e_nx)\ri Q(x)$ in $C^2_{loc}(\BR^3)$ and $Q(x)$ is a local minimizer of the functional $I(Q)=\int\{ \f12|\na Q|^2+f_b(Q)\}\,dx$.
 \item (Theorem \ref{uniquenessthm}, Theorem \ref{match in out} )$Q(x)\ri s_+(n(x)\otimes n(x)-\f13 \mathrm{Id})$ as $|x|\ri\infty$, where $n(x)=T(\f{x}{|x|})$ with $T\in O(3)$ is determined by the asymptotic profile of $n_*$ near $x_0$.
 \item (Theorem \ref{conv on shrinking domain})Let $B_r(x_0)$ be a small neighborhood of $x_0$ that doesn't contain other singularities of $n_*$. Then for any sequence $R_n\uparrow \infty$ and satisfying $R_n\e_n<r$, there holds
 \beqo
 \lim\limits_{n\ri\infty} \left( \sup_{R_n\e_n\leq |x| \leq r} |Q_{\e_n}(x_n+x)-Q_*(x_0+x)| \right)=0,
 \eeqo
 which implies the uniform convergence of $Q_{\e_n}$ to $Q_*$ outside shrinking domains.
\end{itemize}
\end{theorem}

Our results further improve the convergence results in \cite{mz,nz} by showing that the minimizer $Q_{\e_n}$ of the Landau-de Gennes model can be well approximated by the Oseen-Frank minimizer $Q_*$ outside the $O(\e_n)$ neighborhood of the point defect (such neighborhood can be regarded as the defect core). The blow-up limit $Q$ contains the information of the uniform structure of the defect core and its asymptotic behavior at infinity is inherited from the profile near the singularity of $Q_*$. The arguments essentially follow \cite{mp} by Millot-Pisante, which focuses on the similar problem concerning local minimizers for 3-D Ginzburg-Landau functional. However, there are several major differences from our arguments and those of \cite{mp}. On the one hand, the tensor structure gives rise to significant difficulty in our analysis. On the other hand, in \cite{mp} the quantification results of the defect measure from \cite{lw1,lw2} play a crucial role in the proof of strong $H^1$ convergence (see \cite[Proposition 3.1, Proposition 4.1]{mp}), while in this paper we rule out the possible defect measure and obtain strong $H^1$ convergence of  blow-up/blow-down sequences in a more direct way by simply using minimality and the Luckhaus' Lemma (see the proofs of Lemma \ref{strong convergence of Vn}, Theorem \ref{tangentmap} and Lemma \ref{decay radial}).

Our study was motivated by \cite{virga} where numerical investigations indicated that the behaviour near the singularity of the limiting harmonic maps has a universal profile, that is independent of the boundary conditions or the geometry of the domain. This universal profile has an outter part, resembling a so-called hedgehog pattern, and an inner part that has axial symmetry. Our investigation is capable of providing a rigorous interpretation of the studies in \cite{virga} in what concerns the outter part. Stuyding analytically the universal features in the inner part  seems to be a significant analytical challenge.
 
A complete characterization for the behavior of a global minimizer $Q_\e$ inside the defect core is still open. Many research works focus on several typical configurations of the defect core and their stability. Among them, the radial hedgehog solution with the form $Q(x)=r(x)(\f{x}{|x|}\otimes \f{x}{|x|}-\f13{I_d})$ is most extensively studied (see for example \cite{ss,gm1999,ma2012,lamy2013,insz1,insz2}). This configuration is uniaxial everywhere and vanishes at the origin. However, in certain parameter regime the radial hedgehog becomes unstable and biaxiality has to appear near the defect core. Such phenomenon is called ``biaxial escape" and one can refer to \cite{insz2,cl2017,hmp} for rigorous interpretations of this phenomenon within Landau-de Gennes theory at low temperature regime. There are mainly two types of biaxial core structure: the half-degree ring disclination and the split-core solution. These two biaxial configurations have been discovered and studied numerically \cite{mg2000,hqz} and recently rigorously constructed in \cite{yu2020,dmp2,ty} in the axially symmetric setting.

The article is organized as follows. In Section \ref{preliminary estimates}, we introduce the basic mathematical setting of our problem and recall some previous results and estimates that will be used in the rest of the paper. In Section \ref{convergence of blow-up maps}, we study the properties of $Q_{\e_n}$ near a small neighborhood $B_{r_n}(x_0)$ of the singular point $x_0$ and establish the existence of the blow-up limit $Q$ in Proposition \ref{oqconverge}. In Section \ref{tangentmap at infty} we study the behavior of the blow-up limit $Q(x)$ when $|x|\ri\infty$ by proving its tangent map at infinity is just the asymptotic profile of $Q_*$ at $x_0$. The proof is separated into several steps. We first show that there exists a homogeneous degree-1 tangent map of $Q$ at infinity; then we prove the uniqueness of the tangent map; for the last step we show this unique tangent map has to coincide with the hedgehog configuration of $Q_*$ near $x_0$. Finally in Section \ref{shrinking domain} we establish the uniform convergence of $Q_{\e_n}(x_n+x)$ to $Q_*(x_0+x)$ in varing domains $B_r\setminus B_{R_n\e_n}$ for any $R_n\ri\infty$.

\section{Mathematical formulation and preliminary estimates} \label{preliminary estimates}
Let $\Omega$ be an open bounded simply-connected domain in $\BR^3$. We denote by $\mathcal{Q}_0$ the set of traceless symmetric $3\times 3$ matrices, i.e.
\beqo
\mathcal{Q}_0:=\{Q\in \mathcal{M}^{3\times 3},\; Q=Q^T\}.
\eeqo
Consider a Landau-de Gennes functional of the form
\beq\label{Landau-deGennes}
I_\e(Q,\Om)=\int_{\Om}\left[ \f{1}{2}|\na Q|^2+\f{1}{\e^2} f_b(Q)\right]\,dx,\quad Q\in H^1(\Om,\mathcal{Q}_0),
\eeq
with the Dirichlet boundary condition
\beq\label{bdy-con}
 Q|_{\pa\Om}=Q_b=s_+(n_b\otimes n_b-\f13 \mathrm{Id}),\qquad n_b\in C^\infty(\pa\Om,\mathbb{S}^2).
\eeq
That is to say, $Q_b$ is a smooth function taking values in $\mathcal{N}$ which is defined later in \eqref{def of min manifold}.

When $\e=1$, we write
\beq\label{def of I}
I(Q,\Om):=\int_{\Om}\left[ \f{1}{2}|\na Q|^2+ f_b(Q)\right]\,dx
\eeq
The bulk potential is of the form
\beq\label{bulkpotential}
f_b(Q)=-\f{a^2}{2}\mathrm{tr}(Q^2)-\f{b^2}{3}\mathrm{tr}(Q^3)+\f{c^2}{4}[\mathrm{tr}(Q^2)]^2+C.
\eeq
where $C$ is the constant that ensures $\inf\limits_{Q\in\mathcal{Q}_0}f_b(Q)=0$.

We introduce the notion of local minimizers of the energy in the following sense.
\begin{definition}\label{def local minimizer}
Let $Q\in H^1_{loc}(D_0,\mathcal{Q}_0)$ for some domain $D_0\subseteq \BR^3$ ($D_0$ could be $\BR^3$). $Q$ is said to be a local minimizer of $I(\cdot,D_0)$ if
\beqo
I(Q,D)\leq I(V,D)
\eeqo
for any bounded open set $D\Subset D_0$ and $V\in  H^1_{loc}(\Omega,\mathcal{Q}_0)$ such that $Q-V\in H_0^1(D,\mathcal{Q}_0)$.
\end{definition}

The Euler-Lagrange equation for the functional $I_\e$ is given by
\beq\label{eleq}
\D Q_\e=\f{1}{\e^2}(-a^2Q_\e-b^2[Q_\e^2-\f13\mathrm{tr}(Q_\e)^2\,\mathrm{Id}]+c^2\mathrm{tr}(Q_\e)^2Q_\e),
\eeq
where the term $\f{1}{3\e^2} b^2\mathrm{tr}(Q_\e)^2\,\mathrm{Id}$ is a Lagrange multiplier that accounts for the tracelessness constraint.

It is well-known that the bulk potential $f_b$ takes its minimum value on a sub-manifold of $\mathcal{Q}_0$ defined by
\beq\label{def of min manifold}
\mathcal{N}=\{Q=s_+(n\otimes n-\f13 \mathrm{Id}),\; n\in\mathbb{S}^2\},\quad s_+=\f{b^2+\sqrt{b^4+24a^2c^2}}{4c^2}.
\eeq

In \cite{mz,nz}, it was shown that subjected to the Dirichlet boundary condition \eqref{bdy-con}, up to some subsequence, the minimizers $Q_\e$ of $I_\e$ converge to the minimizer $Q_*$ of the functional
\beq\label{harmonicmapenergy}
I_*[Q]=\int_{\Om} |\na Q|^2\,dx,\quad Q\in H^1(\Om,\mathcal{N}), \; Q=Q_b\text{ on }\pa\Om
\eeq
By direct calculation, we have $|\na Q|^2= 2s_+^2|\na n|^2$ for $Q=s_+(n\otimes n-\f13 \mathrm{Id})$. It follows that on the simply-connected domain $\Omega$ this $Q_*$ can be written as $Q_*(x)=s_+(n_*(x)\otimes n_*(x)-\f13\mathrm{Id})$ where $n_*(x)\in H^1( \Om, \BS)$ is a minimizing harmonic map. More precisely, the following results was proved in \cite{mz,nz}.
\begin{theorem}\label{convergethm}
Let $\Om$ be an open bounded simply-connected subset of $\BR^3$ and $Q_\e$ be a minimizer of the minimization problem \eqref{Landau-deGennes}--\eqref{bdy-con}. For any sequence $\e_k\ri 0$, there exists a subsequence, still denoted by $\e_k$, such that $Q_{\e_k}$ converges strongly in $H^1$-norm to a minimizer $Q_*$ of the \eqref{harmonicmapenergy}. Let $\mathrm{Sing}(Q_*)$ denote the singular set of $Q_*$, then
\beqo
Q_{\e_k}\ri Q_*\text{ in }C^{j}_{loc}(\Om\backslash \mathrm{Sing}(Q_*),\mathcal{Q}_0),\;\;\forall j\geq 1.
\eeqo
\end{theorem}

The above theorem gives a nice convergence result of $Q_{\e_k}$ to $Q_*$ away from the singular set $\mathrm{Sing}(Q_*)$. In this note we would like to investigate the behavior of $Q_{\e_k}$ near $\mathrm{Sing}(Q_*)$.

For the limiting harmonic map $Q_*=s_+(n_*\otimes n_*-\f13\mathrm{Id})$, we recall the classical result of Schoen-Uhlenbeck \cite{su} and Brezis-Coron-Lieb \cite[Theorem 1.2]{bcl} that the singular set $\mathrm{Sing}(Q_*)=\mathrm{Sing}(n_*)$ is a set of finitely many isolated points, and near each singular point $x_0$, one has
\beqo
\lim\limits_{r\ri 0}n_*(r(x-x_0))= T\frac{x-x_0}{|x-x_0|},
\eeqo
for some $T\in O(3)$. The convergence is strongly in $H^1(B_1)$ and uniformly in any compact subset of $B_1\backslash \{0\}$. Moreover, using the technique of integrability of a Jacobi field ( see for instance \cite[Theorem 6.3]{simon}), the convergence rate can be controlled by a positive power of $r$,
\beq\label{Jacobi}
\left|n_*(x_0+x)-T\f{x}{|x|}\right|\leq C|x|^{\al}, \quad \forall |x|<r_0.
\eeq
Here $C>0,r_0>0,\al\in(0,1)$ are all positive constants depending just on $n^*$ and $x_0$.

Also there are two basic ingredients in our analysis, which are the monotonicity formula and the small energy regularity estimate, which are both established in \cite{mz}. We list them below.
\begin{lemma}\label{monotonicitylemma}
(Monotonicity lemma, \cite[Section 4, Lemma 2]{mz}) Let $Q_\e$ be a global minimizer of $I_\e$, then
\beq\label{monotonicityformula}
\f{\pa}{\pa R}(\f{1}{R}\int_{B_R}\f{1}{2}|\na Q_\e|^2+\f{1}{\e^2}f_b(Q_\e)\,dx)=\f{1}{R}\int_{\pa B_R}\left|\f{\pa Q_\e}{\pa r} \right|^2\,d\sigma+\f{2}{R^2}\int_{B_R}\f{f_b(Q_\e)}{\e^2}\,dx
\eeq
\end{lemma}

Now we define
\beq\label{def energy density}
e_\e(Q):=\f{1}{2}|\na Q|^2+\f{1}{\e^2}f_b(Q),
\eeq
which denotes the energy density of the Landau-de Gennes functional \eqref{Landau-deGennes} for $Q\in H^1(\Om, \mathcal{Q}_0)$. The following small energy argument holds.
\begin{lemma}\label{smallenergy}
(Small energy regularity, \cite[Section 4, Lemma 7]{mz}) Let $\Om_{\e_k}$ be global minimizers of \eqref{Landau-deGennes}--\eqref{bdy-con} with coefficient $\e_k$ and suppose $Q_{\e_k}\ri Q_*$ in $H^1(\Om)$. Let $K\subset \Om$ be a compact set which contains no singularity of $Q_*$. There exists $C_1>0,\,C_2>0, \,\e_0>0$ such that  for $a\in K$, $0<r<\mathrm{dist}(a,\pa K)$, $\e_k<\e_0$ we have
\beqo
\f{1}{r}\int_{B_r(a)}e_{\e_k}(Q_{\e_k})\,dx\leq C_1,
\eeqo
then
\beqo
r^2\sup\limits_{B_{r/2}(a)}e_{\e_k}(Q_{\e_k})\leq C_2.
\eeqo
\end{lemma}

\section{Convergence of blow-up maps}\label{convergence of blow-up maps}

Take $\e_k\ri 0$ to be the sequence in Theorem \ref{convergethm}, satisfying $Q_{\e_k}\ri Q_*$ in $C^{j}_{loc}(\Om\backslash \mathrm{Sing}(Q_*))$. Assume $ \mathrm{Sing(Q_*)}=\{0\}$, and $n_*(x)\sim \f{x}{|x|}$ near the singular point $0$. We write  the hedgehog map as
\beq\label{Phi}
\Phi(x):=s_+\left(\f{x}{|x|}\otimes\f{x}{|x|}-\f13\mathrm{Id}\right)
\eeq
Now we first fix a sequence of radiuses $r_n$ such that
\begin{equation}
\label{def:rn} r_n\ri 0\quad  \text{as }n\ri\infty.
\end{equation}
According to \cite[Corollary 7.12]{bcl} and \cite[Section 8]{simon0}, it holds that
\beqo
\|Q_*(r_nx)-\Phi(x)\|_{C^2(B_{3/2}\setminus B_{1/2})}+\|Q_*(r_nx)-\Phi(x)\|_{H^1(B_1)}\ri 0,\quad \text{as }n\ri\infty.
\eeqo
Once we fix $\{r_n\}$, we can then choose a subsequence of $\{\e_n\}$, still denoted by $\{\e_n\}$, such that
\beq\label{relation rn en}
\f{r_n}{\e_n}\ri \infty\quad  \text{ as }n\ri\infty,
\eeq
 and
\begin{align}
\label{outside rn uniform conv}&\qquad\qquad \|Q_{\e_n}(x)-Q_*(x)\|_{C^2(\Omega\setminus B_{r_n})}\ri 0 ,\quad \text{as }n\ri\infty , \\
\label{linfty H1 conv} &\|Q_{\e_{n}}(r_nx)-\Phi(x)\|_{C^2(B_{3/2}\setminus B_{1/2})}+\|Q_{\e_n}(r_nx)-\Phi(x)\|_{H^1(B_{1})}\ri 0,\quad \text{as }n\ri\infty.
\end{align}
Here the existence of such $\{\e_n\}$ is guaranteed by Theorem \ref{convergethm}. In the rest of this paper, for convenience we will always work with $\{(r_n, \e_n)\}$ that satisfying \eqref{def:rn}, \eqref{relation rn en}, \eqref{outside rn uniform conv} an \eqref{linfty H1 conv}.


We would like to study the convergence property of the sequence of blow-up maps $Q_{\e_n}(\e_nx)$. Set
\begin{equation*}
    R_n:=\f{r_n}{\e_n}.
\end{equation*}
We first look at the rescaled functions
\beq\label{def:un}
U_n(x):=Q_{\e_n}(r_nx)\quad \text{on }B_1(0).
\eeq
Obviously $U_n$ is a local minimizer, in the sense of Definition \ref{def local minimizer}, of the following functional:
\beqo
\int_{B_1} \left\{\f{1}{2}|\na Q|^2+R_n^2f_b(Q)\right\}\,dx
\eeqo
and satisfies $\|U_n(x)-\Phi(x)\|_{L^\infty(\pa B_1)}\ri 0$ as $n\ri\infty$ due to \eqref{linfty H1 conv}. Then the following lemma holds.
\begin{lemma}\label{bulk0}
$\lim\limits_{n\ri \infty}\int_{B_1} R_n^2 f_b(U_n)\,dx=0$.
\end{lemma}
\begin{proof}
Since $Q_*$ is an admissible map for $I_{\e_n}$, we have
\beqo
\int_{\Om} \f{1}{2}|\na Q_{\e_n}|^2+\f{1}{\e_n^2}f_b(Q_{\e_n})\,dx\leq \int_{\Om}\f12 |\na Q_*|^2\,dx.
\eeqo
On the other hand, $Q_{\e_n}$ converges to $Q_*$ strongly in $H^1$, therefore we have
\beqo
\lim\limits_{n\ri\infty} \int_{\Om} \f{1}{\e_n^2} f_b(Q_{\e_n})\,dx=0
\eeqo
A straightforward calculation shows that $\Phi(x)$ satisfies
\beqo
\f{1}{r}\int_{B_r}\f12 |\na \Phi|^2\,dx=8s_+^2\pi,\quad \forall r>0.
\eeqo
Then using the strong convergence of $Q_{\e_n}(x)$ to $Q_*(x)$, we infer that for any $\d>0$, there exists a $r_\d>0$ such that
\beqo
\lim\limits_{n\ri \infty}\f{1}{r_{\d}}\int_{B_{r_\d}} \f12|\na Q_{\e_n}|^2+\f{1}{\e_n^2}f_b(Q_{\e_n})\,dx\leq 8s_+^2\pi+\d.
\eeqo
Combining the definition of $U_n$ and the monotonicity formula \eqref{monotonicityformula}, we deduce that
\begin{align*}
&\lim\limits_{n\ri\infty}\int_{B_1}\f12|\na U_n|^2+R_n^2f_b(U_n)\,dx\\
=&\lim\limits_{n\ri\infty} \f{1}{r_n}\int_{B_{r_n}}\f12|\na Q_{\e_n}|^2+\f{1}{\e_n^2}f_b(Q_{\e_n})\,dx\\
\leq & \lim\limits_{n\ri\infty} \f{1}{r_\d}\int_{B_{r_\d}} \f12|\na Q_{\e_n}|^2+\f{1}{\e_n^2}f_b(Q_{\e_n})\,dx\\
\leq & 8s_+^2\pi+\d
\end{align*}
Letting $\d\ri 0$ yields
\beq\label{Un8pi}
\lim\limits_{n\ri\infty}\int_{B_1}\f12|\na U_n|^2+R_n^2f_b(U_n)\,dx\leq 8s_+^2\pi.
\eeq
On the other hand, by the strong $H^1$ convergence \eqref{linfty H1 conv} we have
\beqo
\lim\limits_{n\ri\infty}\int_{B_1}\f12|\na U_n|^2\,dx=8s_+^2\pi,
\eeqo
which together with \eqref{Un8pi} implies that
\beqo
\lim\limits_{n\ri \infty}\int_{B_1} R_n^2 f_b(U_n)\,dx=0.
\eeqo
\end{proof}

Now we have the following proposition which says that when $n$ is sufficiently large, the set of points where $U_n(x)$ leaves $\mathcal{N}$ has small measure and concentrates near the origin.
\begin{proposition}\label{diam_estimate}
For any $\d>0$, there exist constants $C_\d$ and $n_\d$ such that
\beq\label{smallness1}
\sup\{|x|, x\in \{\dist(U_n,\cn)\geq\delta\}\}=o(1)\quad \text{as }n\ri\infty,
\eeq
\beq\label{smallness2}
\mathrm{diam}(\{\mathrm{dist}(U_n,\cn)\geq\d\})\leq C_\d R_n^{-1}\quad \forall n\geq n_\d.
\eeq
\end{proposition}
\begin{proof}
We follow closely the proof of \cite[Proposition3.2]{mp}. First recall that \eqref{linfty H1 conv} gives the strong $H^1$ convergence of $U_n$ to $\Phi(x)$ in $B_1$. Now we fix $\d\in (0,1)$ and first prove \eqref{smallness1}. Define
\beq\label{dndelta}
D_n^\d:=\{x\in\bar{B}_1:\dist(U_n(x),\cn)\geq \d\}
\eeq
It suffices to show for any given $r<1$, the set $D_n^\d\subset B_r$ for every $n$ sufficiently large. Note that according to \eqref{linfty H1 conv}, $(B_{1}\backslash B_{1/2})\cap D_n^\d=\varnothing$ as $n\ri\infty$, so we only need to focus on $(\overline{B}_{1/2}\backslash B_{r})\cap D_n^\d$. Since $\Phi(x)$ is smooth outside the origin, we can find a $r_0<\f{r}{8}$ such that
\beqo
\f{1}{r_0}\int_{B_{r_0}(x)}|\na \Phi(x)|^2\,dx\leq \f12 C_3\quad \forall x\in\overline{B}_{1/2}\backslash B_r
\eeqo
Here $C_3$ is a small number to be determined. The strong convergence of $U_n$ to $\Phi$ in $H_1$ implies that
\beqo
\f{1}{r_0}\int_{B_{r_0}(x)}|\na U_n|^2\,dx\leq  C_3\quad \forall x\in\overline{B}_{1/2}\backslash B_r,\;\forall n\geq N_1
\eeqo
where $N_1$ is a large constant. Then we infer from the Lemma \ref{smallenergy} and Lemma \ref{bulk0} that, if $C_3$ is chosen to be suitably small,
\beqo
r_0^2\sup\limits_{B_{r_0/2}(x)}e_{R_n^{-1}}(U_n)\leq C_4,\quad \forall x\in \overline{B}_{1/2}\backslash B_r
\eeqo
where $C_4$ is another constant independent of $n$. By Arzela-Ascoli lemma the sequence $\{U_n\}$ is compact in $L^\infty(\bar{B}_{1/2}\backslash B_r)$, and therefore $\dist(U_n(x),\cn)\ri 0$ uniformly in $B_1\backslash B_r$. In particular, $D_n^\d\subset B_r$ for sufficiently large $n$.

Next we show that there exists a $C_\d$ such that $\mathrm{diam}(D^\d_n)\leq C_\d R_n^{-1}$ for large enough $n$. First we choose a small constant $\d_0>0$ such that, there exists a smooth orthogonal projection of $\cn_{\d_0}$ onto $\cn$. Here $\cn_{\d_0}\subset \mathcal{Q}_0$ denotes the $\d$-neighborhood of $\cn$ in $\mathcal{Q}_0$. We denote $\BP$ as the orthogonal projection from $\cn_{\d_0}$ onto $\cn$. It suffices to show \eqref{smallness2} for all $\d<\d_0$.

We fix $\d<\d_0$ and argue by contradiction.  Let $d_n:=\mathrm{diam}(D_n^\d)$ and suppose $\mu_n:=d_nR_n\uparrow\infty$. We take $a_n,b_n\in D_n^\d$ such that $|a_n-b_n|=d_n$. By \eqref{smallness1} we have $\max\{|a_n|,|b_n|\}\ri 0$. Define
\beqo
c_n:=\f{a_n+b_n}{2},\quad s_n:=\sup\{|x-c_n|:x\in D_n^\d\}.
\eeqo
One can easily verify that $s_n\in[\f{d_n}{2}, d_n)$ from definitions.

We perform the following rescaling
\beqo
V_n(x):=U_n(d_nx+c_n),\quad x\in B_2.
\eeqo
Note that $V_n$ is well-defined for all large enough $n$ due to the fact that $B_{2d_n}(c_n)\subset B_1$ when $n$ is sufficiently large. By the relationship of $s_n$ and $d_n$, we have
\beqo
V_n(x)\in \mathcal{N}_{\delta}, \quad \forall x\in B_2\setminus \bar{B}_1, \; n \text{ sufficiently large.}
\eeqo
Moreover, $V_n(x)$ minimizes the energy
\beqo
\int_{B_2}\f12 |\na Q|^2+\mu_n^2f_b(Q)\,dx.
\eeqo
By the definition of $V_n,U_n$, \eqref{Un8pi} and the monotonicity formula, for every $x_0\in B_2,\; R\in (0,2-|x_0|)$, we have
\begin{align}
\nonumber &\lim\limits_{n\ri\infty}\f{1}{R}\int_{B_{R(x_0)}} e_{\mu_n^{-1}}(V_n)\,dx\\
\label{Venergy}\leq & \lim\limits_{n\ri\infty}\f{1}{1-|d_nx_0+c_n|} \int_{B_{1-|d_nx_0+c_n|}(d_nx_0+c_n)}  e_{R_n^{-1}}(U_n)\,dx\\
\nonumber \leq &\lim\limits_{n\ri\infty} \f{8s_+^2\pi}{1-|d_nx_0+c_n|}=8s_+^2\pi.
\end{align}

Denote $P_1^n:=\f{a_n-c_n}{d_n}$ and $P_2^n:=\f{b_n-c_n}{d_n}$. Up to a rotation we assume
\beqo
P_1^n=P_1=(\f12,0,0),\quad  P_2^n=P_2=(-\f12,0,0).
\eeqo

It is well known, via Chen-Struwe\cite{cs} and Chen-Lin\cite{cl}, that up to a subsequence, $V_n$ converges weakly in $H^1(B_2,\mathcal{Q}_0)$ and strongly in $L^2(B_2,\mathcal{Q}_0)$ to a weakly harmonic map $V\in H^1(B_2,\cn)$. Moreover, there exists a nonnegative Radon measure $\nu$ on $\Om$ such that
\beqo
e_{\mu_n^{-1}}(V_n)\,dx\ri \f{1}{2}|\na V|^2\,dx+\nu \text{ in }B_2.
\eeqo

We have the following lemma.
\begin{lemma}\label{strong convergence of Vn}
$\nu(B_2)=0$, $V_n\ri V$ strongly in $H^1_{loc}(B_2,\mathcal{Q}_0)$ and $V$ is a minimizing harmonic map.
\end{lemma}

We first admit Lemma \ref{strong convergence of Vn} and proceed with the proof of Proposition \eqref{diam_estimate}. A direct consequence of Lemma \ref{strong convergence of Vn} is
\beq\label{fbVnconverge}
\lim\limits_{n\ri\infty}\int_{B_2}\mu_n^2f_b(V_n)\,dx=0
\eeq

For the limiting map $V$, first we claim that
\beqo
\lim\limits_{R\ri 0}\f{1}{R}\int_{B_R(P_i)}|\na V|^2>0,\quad \text{for }i=1,2.
\eeqo
Otherwise if $\lim\limits_{R\ri 0}\f{1}{R}\int_{B_R(P_i)}\na |V|^2=0$ for $i=1$ or $2$ (we assume $i=1$ without loss of generality), then by strong $H^1$ convergence and \eqref{fbVnconverge} we infer that there exist $R_0>0$ and $N_0$ such that for any $n\geq N_0$,
\beqo
\int_{B_{R_0}(P_1)} e_{\mu_n^{-1}}(V_n)\,dx\leq C
\eeqo
for some suitably small constant $C$. Invoking Lemma \ref{smallenergy}, we conclude that the there exists a constant, still denoted by $C$, such that
\beqo
R_0^2\sup\limits_{B_{R_0/2}(P_1)} e_{\mu_n^{-1}}(V_n)\leq C, \quad \forall n\geq N_0
\eeqo
Again by Arzela-Ascoli lemma we have $V_n$ converges uniformly to $V$ in $B_{R_0/2}(P_1)$, which contradicts with the assumption $\dist(V_n(P_1),\cn)=\d>0$. So we get the claim.

On the other hand, since $V$ is a stationary harmonic map, by the quantization results in \cite[Corollary 1]{lr},
\beqo
\lim\limits_{R\ri 0}\f1R \int_{B_R(P_i)}|\na V|^2\,dx=16s_+^2k_i\pi,\quad i=1,2;\; k_i \text{ is a positive integer}.
\eeqo
Recall that from \eqref{Venergy} we have
\beqo
\f{1}{R}\int_{B_R(P_i)}\f12 |\na V|^2\,dx\leq 8s_+^2\pi.
\eeqo
It follows that $k_1=k_2=1$. And by monotonicity formula we have
\beq\label{energynearPi}
\f{1}{R}\int_{B_R(P_i)} |\na V|^2\,dx\geq 16s_+^2\pi,\quad i=1,2; \;R\in(0,1).
\eeq
For every $R\in(0,1)$, denote $Q_R:=(R-\f12,0,0)$. By \eqref{Venergy} and \eqref{fbVnconverge} we have
\beqo
16s_+^2\pi\geq \int_{B_{1}(Q_R)}|\na V|^2\,dx\geq \left(\int_{B_R(P_1)}+\int_{B_{1-R}(P_2)}\right)|\na V|^2\,dx\geq (R+(1-R))16s_+^2\pi=16s_+^2\pi.
\eeqo
It follows that $|\na V|\equiv 0$ on $B_{1}(Q_R)\backslash(B_R(P_1)\cup B_{1-R}(P_2))$ for every $R\in(0,1)$. Note that
\beqo
B_{1}\backslash \{(x,0,0):-1<x<1\}= \bigcup\limits_{R\in(0,1)} B_{1}(Q_R)\backslash(B_R(P_1)\cup B_{1-R}(P_2)),
\eeqo
we therefore deduce that
\beqo
\int_{B_{1}}|\na V|^2\,dx=0
\eeqo
which clearly contradicts with \eqref{energynearPi}. The proof of \eqref{smallness2} is thus complete.
\end{proof}

\begin{proof}[Proof of Lemma \ref{strong convergence of Vn}]
Fix any radius $\rho_0\in(1,2)$. By Fatou's lemma and Fubini's theorem, there is a radius $\rho\in(\rho_0,2)$ and a subsequence of $\{V_n\}$, which still denoted by $\{V_n\}$, such that
\beq\label{relation of vn and v}
\lim\limits_{n\ri\infty} \int_{\pa B_\rho} |V_n-V|^2 d \mathcal{H}^{2}=:\lim\limits_{n\ri\infty} \e_n=0,\ \int_{\pa B_{\rho}}(e_{\mu_n^{-1}}(V_n)+|\na V|^2)d\mathcal{H}^{2}\leq C(\rho_0)<\infty,\ \forall n\in\mathbb{N}.
\eeq
In order to construct an energy competitor, we need the following  extension Lemma which was first proved by \cite{su} and later by Luckhaus \cite{luckhaus}. Here we present the version in \cite[Lemma 2.2.9]{lw3}.
\begin{lemma}\label{luck}
For $n\geq 2$, suppose $u,v\in H^1(\mathbb{S}^{n-1},\cn)$. Then for $\e\in(0,1)$ there is $w\in H^1(\mathbb{S}^{n-1}\times[1-\e,1], \BR^L)$ such that $w|_{\mathbb{S}^{n-1}\times\{1\}}=u,\; w|_{\mathbb{S}^{n-1}\times\{1-\e\}}=v$,
\beqo
\int_{\mathbb{S}^{n-1}\times[1-\e,1]}|\na w|^2\leq C\e\int_{\mathbb{S}^{n-1}}(|\na_T u|^2+|\na_T v|^2)+C\e^{-1}\int_{\mathbb{S}^{n-1}}|u-v|^2,
\eeqo
and
\beqo
\mathrm{dist}^2(w(x),N)\leq C\e^{1-n}\left(\int_{\mathbb{S}^{n-1}}(|\na_T u|^2+|\na_T v|^2)\right)^{\f12}\left(\int_{\mathbb{S}^{n-1}}|u-v|^2\right)^{\f12}+C\e^{-n}\int_{\mathbb{S}^{n-1}}|u-v|^2
\eeqo
for a.e. $x\in \mathbb{S}^{n-1}\times[1-\e,1]$. Here $\na_T$ is the gradient on $\mathbb{S}^{n-1}$ and $\mathbb{R}^L$ is the space in which the manifold $\mathcal{N}$ is embedded.
\end{lemma}

let $W\in H^1(B_2,\cn)$ satisfying $W=V$ on $B_2\setminus B_{2-\rho_0}$. We define the energy competitor $W_n\in H^1(B_2, \mathcal{Q}_0)$ as
\beq\label{def of Wn}
W_n(x):=\begin{cases}
V_n(x), & x\in L_1:= B_2\setminus B_\rho,\\
\f{|x|-\rho+\sqrt{\e_n}}{\sqrt{\e_n}}V_n(\f{\rho x}{|x|})+ \f{\rho-|x|}{\sqrt{\e_n}}\mathbb{P}(V_n(\f{\rho x}{|x|})), & x\in L_2:=B_{\rho}\setminus B_{\rho-\sqrt{\e_n}},\\
\mathbb{P}(K_n(x)), & x\in L_3:= B_{\rho-\sqrt{\e_n}}\setminus B_{\rho-\sqrt{\e_n}-\e_n^{1/6}},\\
W(\f{\rho}{\rho-\sqrt{\e_n}-\e_n^{1/6}}x), & x\in L_4:=B_{\rho-\sqrt{\e_n}-\e_n^{1/6}}.
\end{cases}
\eeq
Here $K_n(x)$ is the connecting function obtained by applying Lemma \ref{luck} to $\mathbb{P}(V_n|_{\pa B_\rho})$ and $V|_{\pa B_\rho}$. To be more precise, $K_n(x)$ satisfies
\begin{align}
\non &K_n(x)=\mathbb{P}(V_n(\f{\rho x}{\rho-\sqrt{\e_n}})) \text{ on } \pa B_{\rho-\sqrt{\e_n}},\\
\non & K_n(x)=V(\f{\rho x}{\rho-\sqrt{\e_n}-\e_n^{1/6}}) \text{ on } \pa B_{\rho-\sqrt{\e_n}-\e_n^{1/6}},\\
\label{energy of Kn} &\dist(K_n(x),\cn)\leq C\e_n^{1/6},\quad \int_{L^3} |\na K_n(x)|^2 \leq C\e_n^{1/6}.
\end{align}
Here we used
\beqo
\int_{\pa B_\rho} |\mathbb{P}(V_n)-V|^2\leq 2\int_{\pa B_\rho}(|\mathbb{P}(V_n)-V_n|^2+|V_n-V|^2)\leq 4\int_{\pa B_\rho}|V_n-V|^2=4\e_n
\eeqo
and applied Lemma \ref{luck} with $n=3$ and $\e=\e_n^{\f16}$ in order to get \eqref{energy of Kn}.

Since $\e_n\downarrow 0$ we know that $K_n(x)\in \cn_{\d_0}$ when $n$ is large enough. Therefore $\mathbb{P}(K_n)$ is well-defined on $L_3$. Direct calculation implies that
\begin{align*}
&\left|\int_{B_\rho} e_{\mu_n^{-1}}(W_n)\,dx-\int_{B_\rho}\f{1}{2}|\na W|^2\,dx\right|\\
=&\left|\f{\sqrt{\e_n}+\e_n^{1/6}}{\rho} \int_{B_\rho}\f{1}{2}|\na W|^2\,dx\right|+\left|\int_{L_2} e_{\mu_n^{-1}}(W_n)\,dx \right|+ \left|\int_{L_3}\f12|\na W_n|^2\,dx\right| \\
=&:I_n+II_n+III_n.
\end{align*}
For the first term $I_n$ since $\e_n\ri 0$ we have $\lim\limits_{n\ri\infty} I_n=0$. For the second term we calculate using polar coordinates:
\beq\label{estimate of II2}
\begin{split}
&\int_{L^2} e_{\mu_n^{-1}}(W_n) dx\\
=&\int_{\r-\sqrt{\e_n}}^\rho r^2 \int_{\mathbb{S}^2} \left\{\f12|\na W_n(r\theta)|^2+\mu_n^2 f_b(\f{|x|-\rho+\sqrt{\e_n}}{\sqrt{\e_n}}V_n(\f{\rho x}{|x|})+ \f{\rho-|x|}{\sqrt{\e_n}}\mathbb{P}(V_n(\f{\rho x}{|x|})))\right\}\ d\theta\,dr\\
\leq& \int_{\rho-\sqrt{\e_n}}^\rho r^2\int_{\mathbb{S}^2} \left\{\f12\left| \na \left(V_n(\rho\theta)+\f{r-\rho}{\sqrt{\e_n}}\left(V_n(\rho\theta)-\mathbb{P}(V_n(\rho\theta))\right)\right)\right|^2+C\mu_n^2f_b(V_n(\rho\theta))\right\}\,d\theta\,dr \\
\leq &  \int_{\rho-\sqrt{\e_n}}^\rho r^2\int_{\mathbb{S}^2} \left\{(1+C\f{(r-\rho)^2}{\e_n}) |\na_T V_n(\rho\theta)|^2+ \f{|V_n(\rho\theta)-\mathbb{P}(V_n(\rho\theta))|^2}{\e_n} +C\mu_n^2f_b(V_n(\rho\theta))\right\}\,d\theta\,dr\\
\leq &  \int_{\rho-\sqrt{\e_n}}^\rho r^2\int_{\mathbb{S}^2} \left\{ C|\na_T V_n(\rho\theta)|^2+ \f{|V_n(\rho\theta)-V(\rho\theta)|^2}{\e_n} +C\mu_n^2f_b(V_n(\rho\theta))\right\}\,d\theta\,dr\\
\leq & C \int_{\rho-\sqrt{\e_n}}^\rho \f{r^2}{\rho^2} \int_{\pa B_\r} |\na_T V_n(x)^2|+\f{|V_n(x)-V(x)|^2}{\e_n}+\mu_n^2 f_b(V_n(x))\,dx\,dr\\
\leq& C\sqrt{\e_n}.
\end{split}
\eeq
From the second line to the third line we used the fact that $f_b(Q)$ is comparable to $\dist(Q,\cn)^2$ when $Q\in \cn_{\d}$ with small enough $\d$. From the third line to the fourth line we utilized $\f{d}{dr}V_n(\f{\rho x}{|x|})=0$ and $|\na_T \mathbb{P}(V_n(\rho \theta))|^2\leq \mathrm{Lip}(\mathbb{P})^2 |\na_T V_n(\rho \theta))|^2$. The final estimate comes from \eqref{relation of vn and v}. Taking $n\ri \infty$ in \eqref{estimate of II2} we get $\lim\limits_{n\ri\infty} II_n=0$.

Finally, by \eqref{def of Wn} and \eqref{energy of Kn} we have
\beqo
III_ n=\int_{L_3} e_{\mu_n^{-1}}(W_n)\,dx =\int_{L_3} \f12|\na \mathbb{P}(K_n(x))|^2\,dx\leq \mathrm{Lip}(\mathbb{P})^2 \int_{L_3}|\na K_n|^2\,dx\leq C\e_n^{1/6}.
\eeqo
When $n\ri\infty$, we obtain $\lim\limits_{n\ri\infty} III_n=0$. Then by minimality of $V_n$ we conclude that
\beqo
\int_{B_\rho}\f12|\na W|^2\,dx=\lim\limits_{n\ri\infty} \int_{B_\rho} e_{\mu_n^{-1}}(W_n)\,dx\geq \lim\limits_{n\ri\infty} \int_{B_\rho} e_{\mu_n^{-1}}(V_n)\,dx\geq\int_{B_\rho}\f12|\na V|^2\,dx.
\eeqo
Hence $V\in H^1(B_2,\cn)$ is an energy minimizing harmonic map. Moreover, if we take $W\equiv V$ in $B_2$, then the calculation above implies that
\beqo
\int_{B_\rho}|\na V|^2\,dx\geq \lim\limits_{n\ri\infty}\int_{B_\rho}e_{\mu_n^{-1}}(V_n)\,dx,
\eeqo
which further implies the strong $H^1_{loc}$ convergence of $V_n$ to $V$ and $
\nu=0$.

\end{proof}

Now we claim that there is a $N\in\mathbb{N}$ large enough such that for any $n\geq N$, the set $D_n^{\d_0}$ defined in \eqref{dndelta} is not empty, where $\delta_0$ is defined in the proof of Proposition \ref{diam_estimate}. Indeed, suppose the claim is wrong and we can find $n_j\ri\infty$ such that for every $j$,
\beqo
U_{n_j}(x)\in \cn_{\d_0},\quad \forall x\in \bar{B}_1.
\eeqo
Then $\BP(U_{n_j})$ are smooth maps from $\bar{B}_1$ to $\cn$, and therefore $\mathrm{deg}(\BP(U_{n_j}),\pa B_r)=0$ for any $r\in(0,1)$. On the other hand, using \eqref{linfty H1 conv} we know that when $j$ is large enough,
\beqo
\mathrm{deg}(\BP(U_{n_j}),\pa B_r)=1,\quad \forall r\in(\f12,1),
\eeqo
which yields a contradiction. The claim is proved.

Now we can define $a_n\in B_1$ such that
\beq\label{def:an}
a_n \in D_n^{\d_0} \text{ and } \dist(U_n(a_n),\cn)=\sup\limits_{x\in B_1} \dist(U_n(x),\cn).
\eeq
According to Proposition \ref{diam_estimate}, we have $a_n\ri 0$ as $n\ri\infty$ and $D_n^{\d_0}\subset B_{r_n}(a_n)$ for some $r_n\leq C_{\d_0} R_n^{-1}$. When $U_n(\pa B_r)\subset \mathcal{N}_{\d_0}$, we also define the topological degree of $U_n|_{\pa B_r}$ as
\beqo
\mathrm{deg}(U_n, \pa B_r(x))=\mathrm{deg}(\BP(U_n),\pa B_r(x)),
\eeqo
By the condition \eqref{linfty H1 conv}, the smoothness of the map $U_n$ and the smallness of $a_n$, we infer that
\beq\label{topodegree}
\mathrm{deg}(U_n,\pa B_{r}(a_n))=1 \quad \text{for every }r\in[C_{\d_0} R_n^{-1},1-|a_n|].
\eeq

Then we define
\beq\label{def:qnn}
Q_n(x)=U_n(\f{x}{R_n}+a_n)=Q_{\e_n}(\e_nx+r_na_n) \text{ for } |x|\leq \overline{R}_n:=R_n(1-|a_n|).
\eeq

\begin{proposition}\label{oqconverge}
Let $Q_\e$ be a global minimizer of the minimization problem \eqref{Landau-deGennes}--\eqref{bdy-con}. For $\{Q_n\}$ defined as in \eqref{def:qnn} with $\{(r_n,\e_n)\}$ as in \eqref{def:rn}, \eqref{relation rn en} \eqref{outside rn uniform conv} and \eqref{linfty H1 conv},
there exists a subsequence, still denoted by $\{Q_n\}$, such that
$Q_n(x)\ri Q$ in $C^2_{loc}(\BR^3,\mathcal{Q}_0)$, where the limiting map $Q$ satisfies
\begin{enumerate}
\item $Q$ locally minimizes the functional $I(\cdot,\BR^3)$ (see \eqref{def of I}) in the sense of Definition \ref{def local minimizer}.
\item $\dist(Q(x),\cn)\ri 0$ as $|x|\ri \infty$, $\mathrm{deg}_{\infty}(Q)=1$ and $\f{1}{R}I(Q,B_R)\ri 8s_+^2\pi$ as $R\ri\infty$.
\end{enumerate}
\end{proposition}

\begin{proof}
Note that $Q_n$ satisfies the Euler-Lagrange equation
\beqo
\Delta Q_n=-a^2Q_n-b^2[Q_n^2-\f13\mathrm{tr}(Q_n)^2\,\mathrm{Id}]+c^2\mathrm{tr}(Q_n)^2Q_n\; \text{ in }B_{\overline{R}_n}.
\eeqo
Also \eqref{Un8pi} implies that
\beq\label{oqn8pi}
\lim\limits_{n\ri \infty} \f{1}{\overline{R}_n} I(Q_n,B_{\overline{R}_n})\leq 8s_+^2\pi.
\eeq
Using standard elliptic regularity theory, we can extract a subsequence, still denoted by $Q_n$, that converges to $Q$ in $C^2_{loc}(\BR^3,\mathcal{Q}_0)$. Here $Q$ solves the same equation as $Q_n$ and also inherits the local minimality from $Q_n$.

By the definition of $Q_n$ and \eqref{topodegree}, we have
\beqo
\mathrm{deg}(Q,\pa B_{r_\d})=\lim\limits_{n\ri\infty}\mathrm{deg}(Q_n,\pa B_{r_\d})=1, \quad \forall r_\d>C_\d.
\eeqo
Hence
\beqo
\mathrm{deg}_\infty(Q)=\lim\limits_{r\ri \infty} \mathrm{deg}(Q,\pa B_r)=1.
\eeqo
Similarly, by \eqref{smallness1} and \eqref{smallness2} we deduce that $\dist(Q(x),\cn)\ri 0$ as $|x|\ri \infty$. It only remains to prove the energy estimate $\f{1}{R}I(Q,B_R)\ri 8s_+^2\pi$ as $R\ri\infty$. On the one hand, by the monotonicity formula and \eqref{oqn8pi} we have
\beqo
\f{1}{R} I(Q,B_R)=\lim\limits_{n\ri\infty} \f{1}{R} I(Q_n,B_R)\leq 8s_+^2\pi.
\eeqo
On the other hand, $\dist(Q(x), \mathcal{N})\ri 0$ as $|x|\ri\infty$ implies that for any $\e>0$ there exists a $R_\e$ such that for all $R>R_\e$,
\beqo
\f1R I(Q, B_R)\geq (1-\e)\f{1}{R}I(\mathbb{P}(Q), B_R).
\eeqo
We recall the well-known fact (see e.g. \cite[Section VII]{bcl}) that for $g:\BS\ri\BS$ with $\deg(g)=1$, it holds that $\int_{\BS}\f12 |\na g|^2\geq 4\pi$. Due to $\deg_\infty(Q)=1$, we get
\beqo
\lim\limits_{R\ri\infty}\f1R I(Q, B_R)\geq (1-\e) 8s_+^2\pi.
\eeqo
Take $\e\ri 0$ and we complete the proof of Proposition \ref{oqconverge}.
\end{proof}

\section{Behavior of the limiting map $Q$ at infinity}\label{tangentmap at infty}

In order to understand the limiting map $Q$ in Proposition \ref{oqconverge}, we study its tangent map at infinity. A tangent map for $Q$ is a map $\Psi: \BR^3\ri \mathcal{Q}_0$ obtained as a weak $H^1_{loc}(\BR^3,\mathcal{Q}_0)$ limit of $Q_{R_n}(x):=Q(R_nx)$ for some sequence $R_n\ri \infty$. Let $T_\infty$ denote the set of all possible tangent maps of $Q$ at infinity. $T_\infty$ can be characterized by the following theorem.

\begin{thm}\label{tangentmap}
Let $Q$ be the map defined in Proposition \ref{oqconverge}, then $T_\infty(Q)$ is not empty. Let $\Psi\in T_\infty(Q)$ and assume $Q_{R_n}(x)\ri \Psi$ weakly in $H^1_{loc}(\BR^3,\mathcal{Q}_0)$. Then $Q_{R_n}(x)\ri \Psi$ strongly in $H^1_{loc}(\BR^3)$ and
\beqo
e_{R_n^{-1}}(Q_{R_n})\,dx\ri \f{1}{2}|\na \Psi|^2\,dx
\eeqo
as convergence of Radon measures. Moreover, there exists $T\in O(3)$ such that
\beq\label{psi}
\Psi(x)=s_+(n(x)\otimes n(x)-\f13\mathrm{Id}), \quad n(x)=T(\f{x}{|x|}).
\eeq
\end{thm}

\begin{proof}
Fix a sequence $R_n\uparrow\infty$. For any $R>0$,  by Proposition \ref{oqconverge} we have
\beqo
\lim\limits_{R_n\ri\infty} \f{1}{R}\int_{B_R} e_{R_n^{-1}}(Q_{R_n})\,dx=  \lim\limits_{R_n\ri\infty} \f{1}{RR_n} I(Q,B_{RR_n})=8s_+^2\pi.
\eeqo

Thus $Q_{R_n}$ is bounded in $H^1_{loc}(\BR^3)$ and up to a subsequence, $Q_{R_n}\ri \Psi$ weakly in $H^1_{loc}(\BR^3)$ and strongly in $L^2_{loc}(\BR^3)$. Since for any $R$, $\lim\limits_{n\ri\infty}\int_{B_R} f_b(Q_{R_n})\,dx=0$, using Fatou's lemma we obtain $\Psi(x)$ take values in $\cn$. Also, $Q$ satisfies the monotonicity formula \eqref{monotonicityformula} because $Q$ locally minimizes the functional $I(Q,\BR^3)$. Then \eqref{monotonicityformula} and $\f1R I(Q, B_R)\ri 8s_+^2\pi$ imply that
\beqo
\lim\limits_{R\ri\infty}\int_{\BR^3\setminus B_R} \f{1}{r}\left|\f{\pa Q(x)}{\pa r} \right|^2\,dx=0.
\eeqo
It follows that for any $0<\rho<R<\infty$,
\beqo
\begin{split}
\int_{B_{R}\setminus B_\r } \f{1}{r}  \left|\f{\pa \Psi(x)}{\pa r} \right|^2\,dx \leq & \lim\limits_{n\ri\infty} \int_{B_{R}\setminus B_\r } \f{1}{r}  \left|\f{\pa Q_{R_n}(x)}{\pa r} \right|^2\,dx\\
=& \lim\limits_{n\ri\infty} \int_{B_{RR_n}\setminus B_{\r R_n} } \f{1}{r}  \left|\f{\pa Q(x)}{\pa r} \right|^2\,dx=0.
\end{split}
\eeqo
This implies that $\Psi(x)=\Psi(\f{x}{|x|})$ for $x\neq 0$. Also the topological degree of $\Psi$ on $\mathbb{S}^2$ is 1 since $\mathrm{deg}_\infty(Q)=1$. Moreover, $\dist(Q(x),\cn)\ri 0$ as $|x|\ri\infty$ implies that
\beqo
\lim\limits_{n\ri\infty}\sup\limits_{x\in B_R\setminus B_{R/2}} \dist(Q_{R_n}(x),\cn)=0
\eeqo

All these properties above enable us to exploit the same argument, as in the proof of Lemma \ref{strong convergence of Vn}, to obtain that $Q_{R_n}\ri \Psi$ strongly in $H^1_{loc}(\BR^3)$ where $\Psi\in H^1_{loc}(\BR^3, \cn)$ is a homogeneous energy minimizing harmonic map with degree $1$.

Now by the classical result of Brezis-Coron-Lieb \cite[Theorem 7.3]{bcl} \footnote{\cite[Theorem 7.3]{bcl} is proved for $\BS$-valued maps, however it also holds for the case of $\BR\mathbb{P}^2$-valued maps, see the discussion in \cite[Section VIII-B-c]{bcl}.}, $\Psi=s_+(n(x)\otimes n(x)-\f13\mathrm{Id})$ for $n(x)=T(\f{x}{|x|})$, for some $T\in O(3)$.

\end{proof}

For the limiting map $Q$ in Proposition \ref{oqconverge}, let's define $Q^\sharp:=\BP(Q)$ as the orthogonal projection of $Q$ onto $\cn$, i.e.
\beqo
|Q(x)-Q^\sharp(x)|=\dist(Q(x),\cn)
\eeqo
By Proposition \ref{oqconverge} we know that $Q(x)$ will stay in a small neighborhood of $\cn$ when $|x|$ is sufficiently large. Consequently, $Q^\sharp(x)$ is well-defined for large $|x|$. We also denote
\beqo
D(x):=Q(x)-Q^\sharp(x).
\eeqo

Then we have the following two lemmas, which mostly rely on the estimates in \cite{nz}.

\begin{lemma}\label{estimates_at_infty_1}
For any positive integer $k$, there exists a positive constant $C_k$ such that
\beqo
|\na^k Q|\leq \f{C_k}{|x|^k}, \quad \text{for all }x\in \BR^3.
\eeqo
\end{lemma}
\begin{proof}
We argue by contradiction. Assume the statement is false, then there would be an integer $k$ and a sequence of points $x_n$ such that
\begin{align*}
&R_n:=|x_n|\ri \infty\quad \text{as }n\ri\infty\\
&R_n^k|\na^k Q|\ri\infty\quad \text{as }n\ri\infty.
\end{align*}
For each $n$, we consider $Q_{n}=Q(R_nx)$ as a local minimizer in the sense of Definition \ref{def local minimizer} of the following functional
\beqo
\int_{B_2} \f12|\na Q|^2+R_n^2 f_b(Q) \,dx
\eeqo
in the ball $B_2(0)$. Thanks to Proposition \ref{oqconverge} and Theorem \ref{tangentmap}, we have that up to a subsequence,
\begin{align*}
& \f{x_n}{R_n}\ri \overline{x}\text{ for some }\overline{x}\in \pa B_1,\\
Q_{n}\ri \Psi\text{ strongly in } &H^1(B_2)\text{ for }\Psi=s_+((T\f{x}{|x|})\otimes (T\f{x}{|x|})-\f13\mathrm{Id}),\;\;T\in O(3).
\end{align*}
The strong $H^1$ convergence implies, as in \cite[Proposition 4]{mz} the uniform convergence of $f_b(Q_n)$ to 0, which allows to use \cite[Lemma 6.7]{mz} to get a uniform gradient bound on $Q_n$ which is updated to convergence in the interior in arbitrarily high norms in \cite[Theorem 1]{nz}. Thus we have
\beqo
Q_n\ri\Psi\text{ in } C^{k+1}(B_{3/2}\backslash B_{1/2}).
\eeqo
Then we can derive
\beqo
\infty=\lim\limits_{n\ri\infty} R_n^k|\na^k Q(x_n)|=\lim\limits_{n\ri\infty}|\na^k Q_n(\f{x_n}{R_n})|=|\na^k\Psi(\overline{x})|<\infty,
\eeqo
which yields a contradiction. The proof is complete.
\end{proof}

\begin{rmk}\label{ck_convergence}
As a consequence of Lemma \ref{estimates_at_infty_1}, we can improve the strong $H^1$ convergence in Theorem \ref{tangentmap} to $C_{loc}^k$ convergence, i.e. assume $\Psi\in T_\infty(Q)$ and $Q_{R_n}\ri \Psi$ weakly in $H_{loc}^1(\BR^3,\mathcal{Q}_0)$, then $Q_{R_n}\ri\Psi$ strongly in $C^k(K,\mathcal{Q}_0)$ for any integer $k$ and compact set $K\subset\BR^3\backslash\{0\}$.
\end{rmk}

\begin{lemma}\label{estimates_at_infty_2}
There exist positive constants $R_0$ and $C$ such that for any $|x|>R_0$,
\begin{align}
\label{gradientoqsharp} |\na Q^\sharp(x)|\leq\f{C}{|x|}&, \;\; |\na (Q^\sharp)^{-1}(x)|\leq \f{C}{|x|},\\
\label{estimateD}  |D(x)|\leq \f{C}{|x|^2}&,\;\; |\na D(x)|\leq \f{C}{|x|^3}.
\end{align}
\end{lemma}

\begin{proof}
The proof follows the same strategy of the proof of Lemma \ref{estimates_at_infty_1}. The key is to show for any sequence $R_n\ri\infty$ such that $Q_{R_n}\ri \Psi$ in $H^1(B_2)$, the following estimates holds in $B_{3/2}\backslash B_{1/2}$:
\begin{align}
\label{rescale_gradientoqsharp} |\na Q_{R_n}^\sharp(x)|\leq C&, \;\; |\na (Q_{R_n}^\sharp)^{-1}(x)|\leq C,\\
\label{rescale_estimateD}  |D_{R_n}(x)|\leq \f{C}{|R_n|^2}&,\;\; |\na D_{R_n}(x)|\leq \f{C}{|R_n|^2},
\end{align}
for some constant $C$. Here $D_{R_n}=Q_{R_n}-\qs_{R_n}$.

For the estimates in \eqref{rescale_gradientoqsharp}, we note that it suffices to prove the first estimate, while the second follows from differentiating the identity $Q_{R_n}^\sharp (Q_{R_n}^\sharp)^{-1}=\mathrm{Id}$. For the first one, note that since $Q_{R_n}^\sharp=\mathbb{P}\circ Q_{R_n}$ for some smooth projection map $\mathbb{P}$,
\begin{equation}\label{nabla_oq_sharp}
|\nabla Q_{R_n}^\sharp(x)|\leq \tilde{C}|\nabla Q_{R_n}(x)|,\quad \left|\f{\pa \qs_{R_n}}{\pa r}\right|\leq \tilde{C}\left|\f{\pa Q_{R_n}}{\pa r}\right|, \quad \text{in }B_{3/2}\backslash B_{1/2}.
\end{equation} with the constant $\tilde C$ independent of $n$, and depending only on the distance between $Q_{R_n}(x)$ and the manifold $\cn$. Also, thanks to the proof of Lemma \ref{estimates_at_infty_1}, we have
\beq\label{nabla_oq_rn}
|\na Q_{R_n}|\leq C \text{ in }B_{3/2}\backslash B_{1/2}.
\eeq
\eqref{nabla_oq_rn} and \eqref{nabla_oq_sharp} together imply the first estimate in \eqref{rescale_gradientoqsharp}.

For the estimate \eqref{rescale_estimateD}, we recall the definition for the matrix $X$ in \cite{nz}:
\beqo
X_n:={R_n^2}(Q_{R_n}^2-\f13s_+Q_{R_n}-\f29s_+^2\mathrm{Id}).
\eeqo
According to \cite[Proposition 4]{nz}, we know that there exists a constant $C$ such that
\begin{align}
\label{C1estimate_X}&\quad \|X_n\|_{C^1(B_{3/2}\backslash B_{1/2})}\leq C,\\
\label{X bound by D}& \f{1}{CR_n^2}|X_n|\leq |D_{R_n}|\leq \f{C}{R_n^2}|X_n|,
\end{align}
which yields the first estimates in \eqref{rescale_estimateD}. For the estimate of $|\na D_{R_n}|$, we will utilize the estimate for $|\na X_n|$ to derive an upper bound.

For the sake of convenience, in the rest of the proof we simply write $Q_{R_n},\,D_{R_n}$ as $Q_n,\,D_n$. We also denote
\beqo
Y_n:=\f{X_{n}}{R_n^2}.
\eeqo


Note that $Q_n$ has the decomposition
\beqo
Q_n=\lam_1 e_1\otimes e_1+ \lam_2 e_2\otimes e_2+\lam_3e_3\otimes e_3,
\eeqo
where $\lam_1\geq \lam_2\geq \lam_3$ are three eigenvalues and $e_1,\,e_2,\,e_3$ are the corresponding unit eigenvectors where are orthogonal with each other. Note that $\lam_i,\ e_i$, $i=1,2,3$ also depend on $n$ and on the location $x$ where $Q_n$ is evaluated. For the sake of convenience, we do not indicate this dependence explicitly here.  Since $Q_n$ is very close to $\mathcal{N}$ when $n$ is large (and we only care about the case for large $n$), we can assume roughly $\lam_1\approx \f23 s_+$ and $\lam_i\approx -\f13 s_+$ for $i=2,\,3$. In particular, one can identify the orthogonal projection of $Q_n$ onto $\mathcal{N}$, which is denoted by $Q_n^\sharp$, as
\beqo
\qs_n=s_+ (e_1\otimes e_1-\f13\mathrm{Id}).
\eeqo
For the validity of this expression, one can refer to \cite[Lemma C.1]{gt}. 
Note that since $\lam_1$ is an isolated eigenvalue of $Q_n$, it is well-known, see for instance \cite{n}, that $e_1$ is as smooth as $Q_n$ and we have
\beq\label{est:e1}
|\na e_1|\leq C |\na Q_n|\leq C \quad \text{ in }B_{3/2}\backslash B_{1/2}.
\eeq
By the definitions of $Y_n$ and $X_n$, we compute
\begin{align*}
Y_n&=Q_n^2-\f13s_+Q_n-\f29s_+^2\mathrm{Id}\\
&=(\qs_n+D_n)(\qs_n+D_n)-\f13s_+(\qs_n+D_n)-\f29 s_+^2\mathrm{Id}\\
&=\qs_n D_n+D_n\qs_n+D_n^2-\f13s_+D_n\\
&=s_+(2e_1\otimes e_1-\mathrm{Id})D_n+D_n^2
\end{align*}
where for the third equality we used that $Q_n^\sharp$ is a root of the polynomial equation $Q^2-\f{s_+}{3}Q-\f{2s_+^2}{9}\mathrm{Id}=0$, see for instance \cite[Lemma 1]{nz}. For the fourth equality we used that $D_n$ and $Q_n^\sharp$ commute (as they have a common eigenbases), together with the definition of $\qs_n$.
\begin{equation}\label{est:Yn}
\begin{split}
|\na Y_n|^2 &= \sum_{\al=1}^3\left|\pa_\al \big(s_+(2e_1\otimes e_1-\mathrm{Id})D_n+D_n^2\big)\right|^2\\
 &=\sum_{\al=1}^3 \left| s_+(2e_1\otimes e_1-\mathrm{Id})\pa_\al D_n+2s_+\pa_\al(e_1\otimes e_1)D_n+\pa_\al(D_n^2)  \right|^2\\
 &=\sum_{\al=1}^3 \bigg\{s_+^2|\pa_\al D_n|^2+4s_+^2|\pa_\al(e_1\otimes e_1)D_n|^2+|\pa_\al (D_n^2)|^2\\
 &\quad \qquad+ 4s_+^2(2e_1\otimes e_1-\mathrm{Id})\pa_\al D_n\,:\, \pa_\al(e_1\otimes e_1)D_n+4s_+\pa_\al(e_1\otimes e_1)D_n\,:\,\pa_\al(D_n^2)\\
 &\quad \qquad  +2s_+(2e_1\otimes e_1-\mathrm{Id})\pa_\al D_n\,:\,\pa_\al(D_n^2)\bigg\}\\
&=\left( s_+^2|\na D_n|^2+4s_+^2|\na(e_1\otimes e_1)D_n|^2+|\na (D_n^2)|^2 \right)+S_n.
\end{split}
\end{equation}

Here $S_n$ is defined to be the sum of all the cross terms. Also in passing from the second line to the third line we have used the fact that
\begin{align*}
&|(2e_1\otimes e_1-\mathrm{Id})\pa D_n|^2\\
=&(2e_1\otimes e_1-\mathrm{Id})\pa D_n\,:\,(2e_1\otimes e_1-\mathrm{Id})\pa D_n\\
=&(2e_1\otimes e_1-\mathrm{Id})^2\,:\, (\pa D_n)^2=|\pa_\al D|^2.
\end{align*}
By \eqref{C1estimate_X}, \eqref{X bound by D} and \eqref{est:e1}, there is a constant $C$ such that for $n$ large enough and $x\in B_{3/2}\backslash B_{1/2}$,  we have
\beq\label{est:Sn}
|S_n|\leq \f{C}{R_n^2}|\na D_n|+\f{C}{R_n^2}|\na D_n|^2.
\eeq
We obtain
\beq\label{est:nablaD}
\f{C}{R_n^4} \geq |\na Y_n|^2\geq (s_+^2-\f{C}{R_n^2})|\na D_n|^2-\f{C}{R_n^2}|\na D_n|,
\eeq
where the first inequality is obtained out of \eqref{C1estimate_X} and the definition of $Y_n$ in terms of $X_n$. The second inequality is obtained combining \eqref{est:Yn} and \eqref{est:Sn} for $n$ large enough.
We point out that the constants $C$ in \eqref{est:nablaD} may represent different values, but they are all independent of $x$ and $n$. Finally, it is straightforward to deduce from \eqref{est:nablaD} that, there exist $N$ and a constant $C$ such that,
\beqo
|\na D_n|\leq \f{C}{R_n^2}, \quad \forall n\geq N, \;\forall x\in B_{3/2}\backslash B_{1/2},
\eeqo
and we conclude the proof.



\end{proof}

Now we are in the position to establish a decay estimate for the radial derivative of $Q$. The argument and the proof are presented in the same spirit as \cite[Proposition 5.3]{mp}, and will also frequently utilize the results from \cite{nz}.

\begin{proposition}\label{decay est}
There exist positive constants $R_0$ and $C$ such that for any $R\geq R_0$,
\beq\label{rad_der_decay}
\int_{|x|>R}\f{1}{|x|}\left| \f{\pa Q}{\pa r} \right|^2\,dx\leq C\f{\log{R}}{R^2}.
\eeq
\end{proposition}
\begin{proof}
Firstly we point out that it suffices to prove \eqref{rad_der_decay} for $Q^\sharp$ since $|\na D(x)|\sim |x|^{-3}$ as $|x|\ri \infty$ thanks to Proposition \ref{estimates_at_infty_2}.  By Proposition \ref{oqconverge}, there exists $R_0$ such that $Q^\sharp(x)$ is well defined and $\mathrm{deg}(Q^\sharp,\pa B_R)=1$ whenever $R\geq R_0$. Recall that $Q$ satisfies the Euler-Lagrange equation in $\BR^3$,
\beqo
\D Q=-a^2Q-b^2[Q^2-\f13\mathrm{tr}(Q)^2\,\mathrm{Id}]+c^2\mathrm{tr}(Q)^2Q
\eeqo
We also have the equation for $Q^\sharp$ by \cite[Proposition 2]{nz},
\begin{equation}
\label{ProjEqn::Est1+}
\begin{split}
\Delta Q^\sharp
	=& -\frac{2}{s_+^2}|\nabla Q^\sharp|^2\,Q^\sharp + \frac{2}{s_+}\Big[\sum_{\alpha = 1}^3(\nabla_\alpha Q^\sharp)^2 - \frac{1}{3}|\nabla Q^\sharp|^2\,\mathrm{Id}\Big]\footnotemark\\
		&\quad  - \Big[T^{-1}\big(\frac{1}{s_+}\,Q^\sharp - \frac{2}{3}\,\mathrm{Id}\big)W - W\big(\frac{1}{s_+}\,Q^\sharp - \frac{2}{3}\,\mathrm{Id}\big)T^{-1}\Big],
\end{split}
\end{equation}
\footnotetext{In reference \cite{nz} a different form is used for the first line, namely the expression in (iv) Corollary 1, which is just the equation for the limit harmonic map. The form we use is an equivalent one, which is the form (ii) in Corollary 1.}
where
\begin{equation}\label{ProjEqn::Est2+}
\begin{split}
W&= 2\nabla Q^\sharp\,\nabla[(Q^\sharp)^{-1}\,Q]\,Q^\sharp - 2Q^\sharp\,\nabla[(Q^\sharp)^{-1}\,Q]\,\nabla Q^\sharp\\
		&\qquad\qquad - \frac{1}{s_+}Q\sum_{\alpha=1}^3 (\nabla_\alpha Q^\sharp)^2 + \frac{1}{s_+}\sum_{\alpha=1}^3 (\nabla_\alpha Q^\sharp)^2\,Q,
\end{split}
\end{equation}
\beq\label{ProjEqn::Est3+}
T= Q - \frac{2}{9}s_+\,\mathrm{tr}[(Q^\sharp)^{-1}\,Q]\,\mathrm{Id} + \beta\Big[\frac{1}{s_+}\,Q^\sharp + \frac{1}{3}\,\mathrm{Id}\Big],
\eeq
and $\beta$ is an arbitrary nonzero real number.

We note that because $Q$ is bounded in $L^\infty$ and close to $\cn$ when $|x|>R_0$, we have that there exists a constant $\tilde C$ depending only on $a^2,b^2,c^2$ and how close the $Q$ is to $\cn$ such that
\begin{equation}
|T^{-1}\big(\frac{1}{s_+}\,Q^\sharp - \frac{2}{3}\,\mathrm{Id}\big)|,|\big(\frac{1}{s_+}\,Q^\sharp - \frac{2}{3}\,\mathrm{Id}\big)T^{-1}|\le \tilde C
\end{equation}

Recall the following decomposition
\begin{equation}\label{decom_q}
Q=Q^\sharp+D
\end{equation} and then we have:

\begin{equation}\label{ProjEqn::Est2++}
\begin{split}
W &= 2\nabla Q^\sharp\,\nabla[(Q^\sharp)^{-1}\,D]\,Q^\sharp - 2Q^\sharp\,\nabla[(Q^\sharp)^{-1}\,D]\,\nabla Q^\sharp\nonumber\\
		&\qquad\qquad - \frac{1}{s_+}D\sum_{\alpha=1}^3 (\nabla_\alpha Q^\sharp)^2 + \frac{1}{s_+}\sum_{\alpha=1}^3 (\nabla_\alpha Q^\sharp)^2\,D\nonumber\\
		&\qquad\qquad - \frac{1}{s_+}Q^\sharp\sum_{\alpha=1}^3 (\nabla_\alpha Q^\sharp)^2 + \frac{1}{s_+}\sum_{\alpha=1}^3 (\nabla_\alpha Q^\sharp)^2\,Q^\sharp.
\end{split}
\end{equation}

We claim that the last two terms vanish. Indeed, for any $\alpha=1,2,3$, we have $\nabla_\alpha Q^\sharp\in T_{Q^\sharp} \cn$, see for instance \cite[Lemma 2]{nz} for a characterization of the tangent space and the normal space to $\cn$ at a point $Q$, which are denoted by $T_Q\cn$ and $(T_Q\cn)^\perp$ respectively. And then by \cite[Lemma 3]{nz} we get that $\sum_{\alpha=1}^3 (\nabla_\alpha Q^\sharp)^2\in((T_{Q^\sharp}) \cn)^\perp$. On the other hand the characterization of the space $((T_{Q^\sharp}) \cn)^\perp$ in \cite[Lemma 2]{nz} shows that the elements in this space commute with matrices $Q^\sharp$, hence the last two terms vanish as claimed.

We rewrite the equation \eqref{ProjEqn::Est1+} as
\beq\label{ProjEqn::Est1+Hx}
\Delta Q^\sharp
	=-\frac{2}{s_+^2}|\nabla Q^\sharp|^2\,Q^\sharp + \frac{2}{s_+}\Big[\sum_{\alpha = 1}^3(\nabla_\alpha Q^\sharp)^2 - \frac{1}{3}|\nabla Q^\sharp|^2\,\mathrm{Id}\Big]+H(x)
\eeq
where
\beq\label{def:H}
H(x):=- \Big[T^{-1}\big(\frac{1}{s_+}\,Q^\sharp - \frac{2}{3}\,\mathrm{Id}\big)W - W\big(\frac{1}{s_+}\,Q^\sharp - \frac{2}{3}\,\mathrm{Id}\big)T^{-1}\Big]=O(|x|^{-4}),\quad \text{as }|x|\ri\infty
\eeq
where the last estimate results from Lemma \ref{estimates_at_infty_2} and relation \eqref{ProjEqn::Est2++} (without the last two terms that vanish). We multiply \eqref{ProjEqn::Est1+Hx} by $\f{\pa Q^\sharp}{\pa r}$. It is straightforward to verify that
\beqo
\left\{-\frac{2}{s_+^2}|\nabla Q^\sharp|^2\,Q^\sharp + \frac{2}{s_+}\Big[\sum_{\alpha = 1}^3(\nabla_\alpha Q^\sharp)^2 - \frac{1}{3}|\nabla Q^\sharp|^2\,\mathrm{Id}\Big]\right\}\cdot\f{\pa Q^\sharp}{\pa r}=0.
\eeqo
Thus we have
\beq\label{multiply_dr}
0=(\D\qs-H(x))\cdot\f{\pa\qs}{\pa r}=\f{1}{|x|}\left| \f{\pa\qs}{\pa r} \right|^2+\di{\Phi(x)}-I(x),
\eeq
where
\beqo
\Phi(x):=\na \qs\cdot\f{\pa \qs}{\pa r}-\f12|\na\qs|^2\f{x}{|x|},\qquad I(x):=H(x)\cdot \f{\pa\qs}{\pa r}= O(|x|^{-5}).
\eeqo
Integrating \eqref{multiply_dr} on an annulus $B_{R_2}\backslash B_{R_1}$ for some $R_2>R_1>R_0$ and then performing integration by parts leads to
\begin{align}
\nonumber & \int_{B_{R_2}\backslash B_{R_1}}\f{1}{|x|}\left| \f{\pa\qs}{\pa r} \right|^2\,dx-\f{1}{2}\int_{\pa B_{R_1}}\left| \f{\pa\qs}{\pa r} \right|^2\,d\sigma\\
\label{inte_by_parts}=&\f12\int_{\pa B_{R_2}}|\na_T \qs|^2\,d\s-\f12\int_{\pa B_{R_1}}|\na_T \qs|^2\,d\s-\f{1}{2}\int_{\pa B_{R_2}}\left| \f{\pa\qs}{\pa r} \right|^2\,d\sigma+\int_{B_{R_2}\backslash B_{R_1}}I(x)\,dx,
\end{align}
where $\na_T$ means the tangential gradient on the sphere. Note that by the second inequality of \eqref{nabla_oq_sharp} and the monotonicity formula \eqref{monotonicityformula} we have
\beqo
\int_0^\infty\left(\f{1}{R} \int_{\pa B_R}\left| \f{\pa \qs}{\pa r} \right|^2\,d\s\right) \,dR\leq \tilde{C} \int_0^\infty\left(\f{1}{R} \int_{\pa B_R}\left| \f{\pa Q}{\pa r} \right|^2\,d\s\right)\,dR<\infty.
\eeqo
So we can find a sequence $\{r_k\}_{k=1}^\infty$ such that
\beq\label{est:rk1}
r_k\ri \infty,\quad  \int_{\pa B_{r_k}}\left| \f{\pa \qs}{\pa r} \right|^2\,d\s\ri 0\qquad \text{ as }k\ri\infty
\eeq
Now we use the compactness property of $Q_{r_k}$ (see Theorem \ref{tangentmap} and Remark \ref{ck_convergence}) and the closeness between $Q_{r_k}$ and $\qs_{r_k}$ (see \eqref{estimateD}) to derive that, up to a subsequence,
\beq\label{est:rk2}
\qs_{r_k}|_{\mathbb{S}^2}\ri \Psi(x)|_{\mathbb{S}^2}\;\text{ in }C^1(\mathbb{S}^2,\cn).
\eeq
Combining \eqref{est:rk1}, \eqref{est:rk2} and the fact that $\Psi$ is energy-minimizing among all degree-1 map from $\mathbb{S}^2$ to $\cn$, we get
\beqo
\lim\limits_{k\ri\infty} \f12\int_{\pa B_{r_k}}|\na_T \qs|^2\,d\s-\f12\int_{\pa B_{R_1}}|\na_T \qs|^2\,d\s-\f{1}{2}\int_{\pa B_{r_k}}\left| \f{\pa\qs}{\pa r} \right|^2\,d\sigma\leq 0
\eeqo
Substituting the above inequality into \eqref{inte_by_parts} gives
\beq\label{inequality:R1}
\int_{|x|>R_1}\f{1}{|x|}\left| \f{\pa \qs}{\pa r} \right|^2\,dx-\f12\int_{\pa B_{R_1}}\left| \f{\pa \qs}{\pa r} \right|^2\,d\s\leq \int_{|x|>R_1}I(x)\,dx\quad \forall R_1\ge R_0
\eeq
Write $R_1=r$ and multiply \eqref{inequality:R1} by $2r$ to obtain
\beq\label{inequality:r}
\f{d}{dr} \left( r^2 \int_{|x|>r}\f{1}{|x|}\left| \f{\pa \qs}{\pa r} \right|^2\,dx\right)\leq Cr^{-1},\quad \text{for some }C>0.
\eeq
Here we utilize the estimates $I(x)=O(|x|^{-5})$.

Finally we integrate \eqref{inequality:r} from $R_0$ to any $R>R_0$ to conclude that
\begin{align*}
&R^2\int_{|x|>R}\f{1}{|x|}\left| \f{\pa \qs}{\pa r} \right|^2\,dx\\
\leq & R_0^2\int_{|x|>R_0}\f{1}{|x|}\left| \f{\pa \qs}{\pa r} \right|^2\,dx+\int_{R_0}^R\f{C}{r}\,dr\leq C\log{R},
\end{align*}
which proves \eqref{rad_der_decay}.

\end{proof}

Now we are ready to prove the uniqueness of the tangent map for $Q$.
\begin{thm}\label{uniquenessthm}
Let $Q$ be the limiting map defined in Proposition \ref{oqconverge}. Then the tangent map at infinity is unique, i.e. there exists a $\Psi$ which is of the form \eqref{psi}, such that
\begin{align}
\label{h1conv}&\lim\limits_{R\ri\infty}\|Q_R-\Psi\|_{H^1_{loc}(\BR^3)}=0\\
\label{ckconv}&\lim\limits_{R\ri\infty}\|Q_R|_{\BS}-\Psi|_{\BS}\|_{C^k(\BS)}=0,\quad \forall k\in \mathbb{N}^+
\end{align}
\end{thm}

\begin{proof}
It suffices to prove \eqref{h1conv}, since \eqref{ckconv} is a direct consequence of \eqref{h1conv} and Remark \ref{ck_convergence}. We prove by contradiction. Assume the statement is false, then there would be two harmonic maps $\Psi_1$ and $\Psi_2$ and two sequences of radiuses $\{r^1_{i}\}_{i=1}^\infty$ and $\{r^2_{j}\}_{j=1}^\infty$ satisfying
\begin{align*}
&\qquad\qquad \lim\limits_{i\ri\infty} r^1_i=\lim\limits_{j\ri\infty}r_j^2=\infty ,\\
&Q_{r^1_i}\ri \Psi_1,\quad Q_{r^2_j}\ri \Psi_2\quad \text{in the sense of }C^2_{loc}.
\end{align*}
We take $i_0$ and $j_0$ be the integers such that for any $i\geq i_0$ and $j\geq j_0$,
\begin{align}
\nonumber &\|Q_{r^1_{i}}-\Psi_1\|_{L^2(\BS)}\leq \f{1}{8}\|\Psi_1-\Psi_2\|_{L^2(\BS)},\\
\nonumber &\|Q_{r^2_{j}}-\Psi_2\|_{L^2(\BS)}\leq \f{1}{8}\|\Psi_1-\Psi_2\|_{L^2(\BS)},\\
\label{diff_12}&\|Q_{r^1_i}-Q_{r^2_{j}}\|_{L^2(\BS)}>\f12\|\Psi_1-\Psi_2\|_{L^2(\BS)}.
\end{align}

We fix the $R_0$ as in Proposition \ref{decay est}. For any $R_0\leq R_1<R_2\leq 2R_1$, we compute
\begin{align*}
\int_{\BS}|Q_{R_1}(\s)-Q_{R_2}(\s)|^2\,d\s
\leq & \int_{\BS} \left((R_2-R_1)\int_{R_1}^{R_2} \left|\f{\pa Q(r\s)}{\pa r}\right|^2\,dr\right)   \,d\s\\
\leq & \int_{\BS} \left(\int_{R_1}^{R_2}r \left|\f{\pa Q}{\pa r}\right|^2 \,dr \right) \,d\s\\
=& \int_{R_1<|x|<R_2} \f{1}{|x|}\left| \f{\pa Q}{\pa r} \right|^2\,dx\leq C \f{\log{R_1}}{R_1^2}
\end{align*}
Now we fix $R_1$, and assume $R_2$ be an arbitrary large radius such that $2^kR_1<R_2\leq 2^{k+1}R_1$ for some non-negative integer $k$. Then we have
\begin{align*}
\|Q_{R_1}-Q_{R_2}\|_{L^2(\BS)}&\leq \sum\limits_{i=0}^{k-1}\|Q_{2^iR_1}-Q_{2^{i+1}R_1}\|_{L^2(\BS)}+\|Q_{R_2}-Q_{2^{k}R_1}\|_{L^2(\BS)}\\
& \leq \sum\limits_{i=0}^k \f{\sqrt{\log{2^i R_1}}}{2^i R_1}\leq C\f{\sqrt{\log{R_1}}}{R_1}
\end{align*}

As a consequence, we have that
\beqo
\lim\limits_{i,j\ri\infty}\|Q_{r^1_i}-Q_{r^2_j}\|_{L^2(\BS)}=0
\eeqo
which yields a contradiction with \eqref{diff_12}. The proof is complete.

\end{proof}

Recall our assumption \eqref{linfty H1 conv} at the very beginning, which says
\beqo
\lim\limits_{n\ri\infty}\|Q_{\e_n}(x)-s_+(\f{x}{|x|}\otimes \f{x}{|x|}-\f13 \mathrm{Id})\|_{L^\infty(B_{\f32r_n}\backslash B_{\f12 r_n})}=0
\eeqo
After a change of variable (making $r_na_n$ as the ``central point"), we have
\beqo
\lim\limits_{n\ri\infty}\|Q_{\e_n}(x+r_na_n)-s_+(\f{x+r_na_n}{|x+r_na_n|}\otimes \f{x+r_na_n}{|x+r_na_n|}-\f13 \mathrm{Id})\|_{L^\infty(B_{\f43r_n}\backslash B_{\f23 r_n})}=0,
\eeqo
where $a_n$ is defined in \eqref{def:an}. Note that when $\f23r_n\leq |x|\leq \f43r_n$, $\f{x+r_na_n}{|x+r_na_n|}$ is very close to $\f{x}{|x|}$ given $|a_n|$ sufficiently small (see the remark of $a_n\ri 0$ after \eqref{def:an}). As $n\ri\infty$, we obtain
\beq\label{outer}
\lim\limits_{n\ri\infty}\|Q_{\e_n}(x+r_na_n)-s_+(\f{x}{|x|}\otimes \f{x}{|x|}-\f13 \mathrm{Id})\|_{L^\infty(B_{\f43r_n}\backslash B_{\f23 r_n})}=0
\eeq

On the other hand, by Theorem \ref{tangentmap} and Theorem \ref{uniquenessthm}, we have that
\beq\label{inner1}
\lim\limits_{R\ri\infty} \|Q(x)-s_+(n\otimes n-\f13\mathrm{Id})\|_{L^\infty(B_{2R}\backslash B_{R})}=0
\eeq
where $n(x)=T(\f{x}{|x|})$ for some $T\in O(3)$. Since $Q$ is obtained by taking a $C^2_{loc}$ limit of $Q_n(x)$ (see \eqref{def:qnn}), for any fixed $R$ it holds that
\beq\label{inner2}
\lim_{\substack{\f{r_n}{2\e_n}\geq R,\\ n\ri\infty}} \|Q_{\e_n}(x+r_na_n)-Q(\f{x}{\e_n})\|_{L^\infty(B_{2R\e_n}\setminus B_{R\e_n})}=0
\eeq
\eqref{inner1} and \eqref{inner2} together imply that
\beq\label{inner}
\lim\limits_{R\ri\infty} \left(  \lim_{\substack{\f{r_n}{2\e_n}\geq R,\\ n\ri\infty}}  \|Q_{\e_n}(x+r_na_n)-s_+(n(x)\otimes n(x)-\f13\mathrm{Id})\|_{L^\infty(B_{2R\e_n}\setminus B_{R\e_n})} \right)=0.
\eeq

Comparing \eqref{outer} and \eqref{inner} we knows that $Q_{\e_n}(x+r_na_n)$ is close to $s_+(\f{x}{|x|}\otimes \f{x}{|x|}-\f13\mathrm{Id})$ at $|x+r_na_n|\sim r_n$, but when $|x+r_na_n|\sim R\e_n$ for large enough $R$, it is asymptotically $s_+(n(x)\otimes n(x)-\f13 \mathrm{Id})$ as $n\ri\infty$. A natural question would be whether or not $n(x)=\f{x}{|x|}$, so that the behavior of $Q_{\e_n}$ on the outer sphere $\pa B_{r_n}(r_na_n)$ will match that of the inner sphere $\pa B_{R\e_n}(r_na_n)$. The answer is positive.

\begin{theorem}\label{match in out}
Let $Q$ be the limiting map in Proposition \ref{oqconverge} and $\Psi$ is its unique tangent map at infinity. Then $\Psi=\Phi$, i.e.
\beqo
\Psi=s_+(\f{x}{|x|}\otimes \f{x}{|x|}-\f13\mathrm{Id}).
\eeqo
\end{theorem}

To prove the theorem, we need the following lemma, which gives crucial estimates on the decay rate of the radial derivatives.

\begin{lemma}\label{decay radial}
Let $Q_{\e_n}$ be the sequence of minimizers such that $Q_{\e_n}(\e_nx+r_na_n)\ri Q$ in $C^2_{loc}(\BR^3)$ with $r_n$, $\e_n$ satisfying \eqref{def:rn}, \eqref{relation rn en}, \eqref{outside rn uniform conv} and \eqref{linfty H1 conv}, $a_n$ as defined in \eqref{def:an}. Then there is a positive constant $C$, such that for any $\e_n$ and $R\leq \f{r_n}{2\e_n}$, it holds that
\beq\label{decay of radial der}
\int_{\{R\e_n\leq |x|\leq 2R\e_n\}}\f{1}{|x|} \left|\f{\pa Q_{\e_n}(x+r_na_n) }{\pa r}\right|^2\,dx\leq \f{C}{R^4}.
\eeq
\end{lemma}

\begin{proof}
Without loss of generality, we can assume $a_n=0$ for every $n$, since we only need the property $|a_n|\ri 0$ as $n\ri\infty$ and the exact location of $a_n$ won't affect our proof.

Assume such constant $C$ does not exist. Then we can find a sequence $\e_k,\; R_k\leq \f{r_k}{2\e_k}$ (in this sequence, one $\e_k$ can appear repeatedly) such that
\beq\label{contra:radial deri}
R_k^4\int_{\{R_k\e_k\leq |x|\leq 2R_k\e_k\}} \f{1}{|x|}\left|\f{\pa Q_{\e_k}}{\pa r}\right|^2\,dx\ri\infty,\quad \text{as }k\ri\infty.
\eeq
In order that \eqref{contra:radial deri} holds, we must have
\beq\label{contra:assumption}
\lim\limits_{k\ri\infty} \e_k=0,\qquad \lim\limits_{k\ri\infty}R_k=\infty.
\eeq
Here we briefly justify these two limits. Firstly if $\limsup\limits_{k\ri\infty}\e_k>0$, then there exists an integer $i_0$ such that $\e_k=\e_{i_0}>0$ holds for infinitely many $k$. For all such $k$, $R_k$ is uniformly bounded since we require $R_k\leq \f{r_k}{2\e_k}=\f{r_{i_0}}{2\e_{i_0}}$. Then $R_k^4\int_{\{R_k\e_k\leq |x|\leq 2R_k\e_k\}} \f{1}{|x|}\left|\f{\pa Q_{\e_k}}{\pa r}\right|^2\,dx$ is also bounded, which contradicts with \eqref{contra:radial deri}. Secondly, if $R_k$ doesn't go to infinity, then we assume $\liminf\limits_{k\ri\infty} R_k=R_0<\infty$. By monotonicity formula \eqref{monotonicityformula} we know $\int_{R_k\e_k\leq |x|\leq 2R_k\e_k}\f{1}{|x|}\left|\f{\pa Q_{\e_k}}{\pa r}\right|^2\,dx$ is uniformly bounded. And this will remain bounded after multiplying bounded $R_k^4$, which also contradicts with \eqref{contra:radial deri}. Therefore $R_k$ has to go to infinity.

Now we define
\beqo
P_k:=Q_{\e_k}(R_k\e_k \cdot x).
\eeqo
Then $P_k$ satisfies the following properties.
\begin{enumerate}
  \item $P_k$ minimizes the functional $\int_{B_3}\{\f12|\na Q|^2+R_k^2f_b(Q)\}\,dx$.
  \item\label{second property} For any $r\in(0,3)$, we have
  \beqo
  \lim\limits_{k\ri\infty} \f{1}{r}\int_{B_r}\f12|\na P_k|^2+R_k^2f_b(R_k)=8s_+^2\pi.
  \eeqo
  The fact that the limit on the left-hand side is bounded from above by $8s_+^2\pi$ comes from \eqref{Un8pi} and the monotonicity formula; the lower bound by $8s_+^2\pi$ follows from the $C^2_{loc}$ convergence of $Q_{\e_n}(\e_nx)$ to $Q(x)$, the asymptotic behavior of $Q(x)$ for large $|x|$, and the monotonicity formula as well.
  \item For any $0< r_1<r_2\leq 3$, we have
  \beq\label{homogeneous property}
  \lim\limits_{k\ri\infty}\int_{\{r_1\leq |x|\leq r_2\}}\f{1}{|x|}\left|\f{\pa P_k}{\pa r}\right|^2\,dx=0
  \eeq
  This property follows from Property \ref{second property} and the monotonicity formula \eqref{monotonicityformula}.
  \item For any $r>0$, it holds that
  \beqo
  \lim\limits_{k\ri\infty} \sup\limits_{x\in B_3\setminus B_r} \dist(P_k(x),\cn)=0.
  \eeqo
  This follows from \eqref{smallness2} and $R_k\ri \infty$.
\end{enumerate}

All these properties enable us to exploit similar arguments as in the proofs of Lemma \ref{strong convergence of Vn} and Theorem \ref{tangentmap} to get the strong $H_{loc}^1$ convergence of $P_k$ to $\overline{P}$, where $\overline{P}\in H^1_{loc}(B_3, \cn)$ is a homogeneous minimizing harmonic map of degree 1. In addition $\overline{P}$ has the form
\beqo
\overline{P}=s_+(m(x)\otimes m(x)-\f13\mathrm{Id}), \quad m=T_m\left(\f{x}{|x|}\right) \text{ for some }T_m\in O(3).
\eeqo

We now apply \cite[Proposition 9]{nz} to get
\beqo
\|P_k-\overline{P}\|_{C^2(B_2\backslash B_1)}\leq \f{C}{R_k^2},\quad \text{for some constant }C.
\eeqo
Consequently, we calculate
\begin{align*}
R_k^4\int_{\{R_k\e_k\leq |x|\leq 2R_k\e_k\}} \f{1}{|x|}\left|\f{\pa Q_{\e_k}}{\pa r}\right|^2\,dx&= R_k^4 \int_{B_2\backslash B_1}\f{1}{|x|}\left|\f{\pa P_k }{\pa r}\right|^2\,dx\\
&=R_k^4 \int_{B_2\backslash B_1}\f{1}{|x|}\left|\f{\pa (P_k-\overline{P}) }{\pa r}\right|^2\,dx\leq C,
\end{align*}
which contradicts with \eqref{contra:radial deri}. The proof is complete.

\end{proof}

\begin{proof}[Proof of Theorem \ref{match in out}]
We prove by contradiction. Assume the conclusion is false, then
\beq\label{diff in out}
\|\Psi-s_+(\f{x}{|x|}\otimes\f{x}{|x|}-\f13\mathrm{Id})\|_{L^2(\BS)}=\s>0.
\eeq
Because $\Psi$ is the tangent map of $Q$, we can find a large enough $r_0$, such that
\begin{align*}
&r_0>\f{2C^{1/4}}{\s^{1/2}}, \quad C \text{ is the constant in \eqref{decay of radial der}}\\
&\|Q(r_0x)-\Psi\|_{L^2(\BS)}<\f{\s}{16}.
\end{align*}
Also, recall that $Q$ is obtained by taking limit of $Q_n$ (see \eqref{def:qnn} and Proposition \ref{oqconverge}), we can find a large integer $N_0$, such that for any $n\geq N_0$,
\begin{align}
\label{inner map} &\| Q_{\e_n}(r_0\e_n x+r_na_n)-\Psi(x) \|_{L^2(\BS)}<\f{\s}{8},\\
\label{outer map}&\|Q_{\e_n}(rx+r_na_n)-s_+\left(\f{x}{|x|}\otimes \f{x}{|x|}-\f13\mathrm{Id}  \right)\|_{L^2(\BS)}<\f{\s}{8}\quad \forall r\in[\f12r_n,\f32r_n],
\end{align}
where for \eqref{outer map} we used \eqref{linfty H1 conv} and the fact that $a_n\ri 0$, see the definition \eqref{def:an} of $a_n$ and the discussion right after \eqref{def:an}. Using \eqref{diff in out}, \eqref{inner map} and \eqref{outer map} we have for $n\geq N_0$, $r\in [\f12 r_n,\f32r_n]$,
\beq\label{in out}
\|Q_{\e_n}(rx+r_na_n)-Q_{\e_n}(r_0\e_nx+r_na_n)\|_{L^2(\BS)}\geq \f{3\delta}{4}.
\eeq
We define $k_n$ to be the largest integer such that $k_n\leq |\log_2 (\f{r_n}{r_0\e_n})|$
Following the same argument as in the proof of Theorem \ref{uniquenessthm}, we have
\begin{align*}
&\|Q_{\e_n}(r_0\e_nx+r_na_n)-Q_{\e_n}(2^{k_n}r_0x+r_na_n)\|_{L^2(\BS)}\\
\leq &\sum_{j=0}^{k_n} \|Q_{\e_n}(2^{j}r_0\e_nx+r_na_n)-Q_{\e_n}(2^{j+1}r_0\e_nx+r_na_n)\|_{L^2(\BS)}\\
\leq & \sum_{j=0}^{k_n-1} \left( \int_{\{2^{j}r_0\e_n\leq |x|\leq 2^{j+1}r_0\e_n\}} \f{1}{|x|}\left|\f{\pa Q_{\e_n}}{\pa r}\right|^2\,dx \right)^{1/2}\\
\leq & \sum_{j=0}^{k_n-1} \f{C^{1/2}}{4^jr_0^2}\leq  \f{2C^{1/2}}{r_0^2}<\f{\s}{2}.
\end{align*}
Here we have used Lemma \ref{decay radial} and $r_0>\f{2C^{1/4}}{\s^{1/2}}$. Note that the result of the calculation above already contradicts \eqref{in out}, which completes our proof of Theorem \ref{match in out}.

\end{proof}

\section{Uniform convergence outside shrinking regions}\label{shrinking domain}

Using all the characterizations of the limit map $Q$, we can further prove the following convergence result.

\begin{theorem}\label{conv on shrinking domain}
Let $\Om$ be an open bounded subset of $\BR^3$ and $Q_\e$ be a minimizer of the energy functional $I_\e[Q]$\eqref{Landau-deGennes} with the boundary condition \eqref{bdy-con}. For any sequence $\e_n\ri 0$, one can find a subsequence, still denoted by $\{\e_n\}$, and a sequence of points $\{x_n\}$, such that
\begin{enumerate}
\setlength\itemsep{0.5em}
\item $Q_{\e_n}\ri Q_*$ in $H^1(\Om)$, where $Q_*$ is a minimizer of \eqref{harmonicmapenergy};
\item $x_n\ri x_0$ as $n\ri\infty$, where $x_0\in \mathrm{Sing}(Q_*)$;
\item Let $B_r(x_0)$ be a small neighborhood of $x_0$ that doesn't contain other singularities of $Q_*$. Then for any sequence of radiuses $R_n$ such that $\lim\limits_{n\ri\infty}R_n=\infty$ and $R_n\e_n<r$, there holds
    \beqo
    \lim\limits_{n\ri\infty}\left( \sup_{R_n\e_n\leq |x| \leq r}|Q_{\e_n}(x+x_n)-Q_*(x+x_0)|\right)=0.
    \eeqo
\end{enumerate}
\end{theorem}

\begin{proof}
The proof will use compactness arguments similar to those that haven been applied several times before. Without loss of generality, we assume $x_0=0$ and $Q_{*}(x)\sim \Phi(x)= s_+\left(\f{x}{|x|}\otimes \f{x}{|x|}-\f13\mathrm{Id}\right)$ when $x$ approaches $0$. First the existence of $H^1$ convergent subsequence $Q_{\e_k}$ is given in \cite{mz} (see Theorem \ref{convergethm}). Taking up to a subsequence, we can find a sequence of radii $\{r_n\}$ such that $\{r_n,\e_n\}$ satisfy \eqref{def:rn}, \eqref{relation rn en}, \eqref{outside rn uniform conv} and \eqref{linfty H1 conv}. Let $a_n$ be as defined in \eqref{def:an} and we simply take $x_n=r_na_n$. Then $x_n\ri 0$ is guaranteed by the definition and Proposition \ref{diam_estimate}.

Now we are ready to verify the Property (3) in the theorem. We argue by contradiction. Suppose there exists a subsequence of $\{\e_n\}$, still denoted by $\{\e_n\}$, and a sequence of points $\{y_n\}$ such that
\begin{align}
\nonumber &\qquad \qquad \lim\limits_{n\ri\infty}\f{|y_n|}{\e_n}=\infty,\quad |y_n|\leq r\\
\label{contra:yn} &|Q_{\e_n}(y_n+x_n)-Q_*(y_n)|\geq \delta>0,\quad \text{ for some constant }\delta
\end{align}

First it is obvious that \eqref{contra:yn} implies $|y_n|\ri 0$, otherwise it will conflict with the uniform convergence result of $Q_{\e_n}$ to $Q_*$ on any compact set $K$ that doesn't contain any point in $\mathrm{Sing}(Q_*)$. Thus we get
\beq\label{Q* close to hedgehog}
\lim\limits_{n\ri\infty}|Q_*(y_n)-\Phi(y_n)|=0.
\eeq

Next by \eqref{outside rn uniform conv} we only need to consider the case $y_n+x_n\in B_{r_n}$. Now we define
\beqo
Z_n(x):= Q_{\e_n}(|y_n|x+x_n).
\eeqo

Then by exactly the same argument as in the proof of Lemma \ref{decay radial} to derive the convergence of $\{P_k\}$, we can extract a subsequence, still denoted by $\{Z_n\}$, such that
\begin{align*}
&\quad Z_n(x) \ri \Psi(x)\; \text{ in }H^1(B_2)\cap C^2(B_{3/2}\backslash B_{1/2}),\\
& \Psi(x)=s_+(n\otimes n-\f13\mathrm{Id}), \quad n=T\left((\f{x}{|x|}\right)) \text{ for some }T\in O(3).
\end{align*}
Using Lemma \ref{decay radial} and arguing in the same way as in the proof of Theorem \ref{match in out}, one can easily verify that
\beqo
T\left(\f{x}{|x|}\right)=\f{x}{|x|},\text{ i.e. } \Psi(x)=\Phi(x).
\eeqo
Therefore we have
\beq\label{Zn close to hedgehog}
\lim\limits_{n\ri\infty} |Q_{\e_n}(y_n+x_n)-\Phi(y_n)|=0.
\eeq

Combining \eqref{Q* close to hedgehog} and \eqref{Zn close to hedgehog} yields a contradiction with \eqref{contra:yn}, which completes our proof.

\end{proof}

\bibliographystyle{acm}
\bibliography{uniform_bib}

\end{document}